\documentclass[12pt]{amsart}%
\usepackage{amsfonts}
\usepackage{amsmath}
\usepackage{amssymb}
\usepackage{amsthm}
\usepackage{mathrsfs}
\usepackage{graphicx}%
\usepackage{amsmath}
\usepackage{setspace}
\usepackage{hyperref}
\usepackage{cite}
\usepackage{amsfonts}
\usepackage{algorithmic}
\usepackage{textcomp}
\usepackage{stmaryrd}

\allowdisplaybreaks[4]
\makeatletter
\@addtoreset{equation}{section}
\makeatother

\newtheorem{theorem}{Theorem}[section]

\newtheorem{corollary}[theorem]{Corollary}

\newtheorem{definition}[theorem]{Definition}

\newtheorem{lemma}[theorem]{Lemma}

\newtheorem{proposition}[theorem]{Proposition}
\newtheorem{remark}[theorem]{Remark}

\makeatletter
\def\@makefnmark{}
\makeatother

\begin{document}

\title[Anisotropic minimal surface equation with Dirichlet boundary condition]{Anisotropic minimal surface equation with Dirichlet boundary condition}

\author{Lu Chen}
\address[Lu Chen]{Key Laboratory of Algebraic Lie Theory and Analysis of Ministry of Education, School of Mathematics and Statistics, Beijing Institute of Technology, Beijing
100081, PR China;  Tangshan Research Institute, Beijing Institute of Technology, Tangshan 063000, PR China}
\email{chenlu5818804@163.com}

\author{Jiali Lan}
\address[Jiali Lan]{Key Laboratory of Algebraic Lie Theory and Analysis of Ministry of Education, School of Mathematics and Statistics, Beijing Institute of Technology, Beijing
100081, PR China}
\email{17636268505@163.com}

\author{Haolin Liu}
\address[Haolin Liu]{Key Laboratory of Algebraic Lie Theory and Analysis of Ministry of Education, School of Mathematics and Statistics, Beijing Institute of Technology, Beijing
100081, PR China}
\email{haolinliu2023@126.com}

\address{}

\thanks{The first author was partly supported by the  National Natural Science Foundation of China (No. 12271027) and Hebei Natural
Science Foundation (No. A2025105003). }

\begin{abstract}
This paper investigates the Dirichlet problem for the anisotropic minimal surface equation in a bounded domain. Under the natural assumption of non-negative boundary anisotropic mean curvature, we establish the unique solvability of the Dirichlet problem for   continuous boundary data. To achieve this, an essential a priori gradient estimate is established, which also allows us to prove a weak version of Bernstein's theorem for entire solutions under a sharp, one-sided linear growth assumption. Moreover, using the direct method in the calculus of variations, we prove the existence and local Lipschitz regularity of generalized minimizers in $BV(\Omega)$ with $L^1(\partial\Omega)$ boundary data. We also find that this variational formulation naturally yields a Neumann-type boundary condition, geometrically explaining why the Neumann problem requires no boundary curvature constraints.
\end{abstract}

\maketitle

 \section{Introduction and Main Results}


In this paper, we   consider the graph of a function $u(x)$ defined on an open set $\Omega\subset\mathbb{R}^n$, such a graph is usually referred to as a non-parametric surface. The non-parametric anisotropic area  is given by
\begin{equation}\label{e1.5}\mathcal{A}_{F}(u,\Omega)=\int_{\Omega}F(Du,-1)dx,\end{equation}
where $F:\mathbb{R}^{n+1}\to[0,+\infty)$ is non-negative, convex, and positively homogeneous function of degree one (see Section 2).  The density $F$ geometrically represents a Minkowski norm on $\mathbb{R}^{n+1}$ and physically models the direction-dependent surface tension of an interface. In materials science, this functional arises naturally to describe the surface free energy of crystalline solids, where different crystallographic planes exhibit distinct packing densities and atomic bonds, leading to the celebrated Wulff construction for equilibrium crystal shapes.

\medskip
The anisotropic Bernstein problem asks whether critical points of $\mathcal{A}_{F}$ defined on all of $\mathbb{R}^n$  are necessarily affine functions.  In the case of the area functional $F(\xi) = |\xi|$, it is known through  pioneering works   of Bernstein, Fleming \cite{F62}, De Giorgi \cite{DG65}, Almgren \cite{AF66}, Simons \cite{S68}, and Bombieri-De Giorgi-Giusti \cite{BDGG} that the answer is positive if and only if $n \leq 7$.
For general uniformly elliptic integrands, it is known that the answer is positive in dimension $n = 2$ by  Jenkins \cite{J} and in dimension $n = 3$ by   Simon \cite{S}.
 Recently, Mooney and Yang \cite{MY} constructed nontrivial entire anisotropic minimal graphs for $n=4$, demonstrating  that such non-parametric examples can exhibit genuinely nonlinear behavior in low dimensions.

\medskip
The corresponding Euler-Lagrange equation associated with the anisotropic area functional $\mathcal{A}_F$ is   given by the anisotropic minimal surface equation:
\[\operatorname{div}\left(D_{\xi'}F(Du,-1)\right)=0,\]
where $D_{\xi'}F=(F_{\xi_1},\cdots, F_{\xi_{n}})$. Equivalently, this equation can be rewritten as
\begin{equation}\label{e1.6}F_{\xi_i\xi_j}(Du,-1)u_{ij}=0,\end{equation}
  where we sum over $i,j=1,\dots,n$  according to the Einstein convention. A natural question is that of the existence of solutions of the Dirichlet problem, namely of solutions of the anisotropic minimal surface equation \eqref{e1.6} taking prescribed values on the boundary $\partial\Omega$.

\medskip
For the isotropic case $F(\xi)=|\xi|$, the Euler-Lagrange equation reduces to the  classical minimal surface equation
 \begin{equation}\label{e1.4}\operatorname{div}\left(\frac{Du}{\sqrt{1+|Du|^2}}\right)=0.\end{equation}
 The solvability of its Dirichlet problem has a celebrated history. For $n=2$, results of Bernstein \cite{B}, Haar \cite{H}, Rad\'o \cite{RT}, Finn \cite{F65}, and Finn-Osserman \cite{FO} show that  the Dirichlet problem is well-posed for arbitrary continuous boundary data if and only if $\Omega$ is convex. Finally in 1968, Jenkins-Serrin \cite{JS} proved that the Dirichlet problem in $n$ dimensions is always solvable if the mean curvature of $\partial\Omega$ is nowhere negative, using   interior gradient estimates. The interior gradient estimates for the minimal surface equation was obtained by Finn \cite{F54} for $n=2$ and by Bombieri-De Giorgi-Miranda \cite{BDM} for higher dimensions. For equations with general mean curvature, we refer to Ladyzhenskaya-Uraltseva \cite{LU}, Trudinger \cite{T} and Simon \cite{S76}. All of these methods rely fundamentally on test function arguments and the resulting Sobolev inequalities. A more detailed history can be found in Gilbarg-Trudinger \cite{GT}.

\medskip
The purpose of this paper is to prove the existence of solutions to the Dirichlet problem for the anisotropic minimal surface equation \eqref{e1.6}. To properly define the anisotropic mean curvature of the boundary $\partial\Omega$, we introduce $\hat{F}(p)=F(p,0)$ for $p\in\mathbb{R}^n$. Wang-Xia \cite{WX} calculated the anisotropic mean curvature of a level set of $u$. More precisely, they proved that
\[H_{\hat{F}}(S_t)=\hat{F}_{\xi_i\xi_j}(Du)u_{ij},\]
where $S_t=\{x\in\overline{\Omega}: u(x)=t\}$. In particular, the anisotropic mean curvature of $\partial\Omega$ is
\[H_{F}(\partial\Omega)=F_{\xi_i\xi_j}(Dd,0)d_{ij},\]
where $d(x)$ denotes the distance from $x$ to $\partial\Omega$. Thus the assumption
\[ F_{\xi_i\xi_j}(Dd,0)d_{ij}\le0\]
is exactly the nonnegativity of the anisotropic mean curvature of $\partial\Omega$ with respect to the outward normal.

\medskip

This observation enables us to prove that the anisotropic Dirichlet problem is solvable provided that the anisotropic mean curvature of $\partial\Omega$ is everywhere non-negative.   Throughout this paper, we further assume $F \in C^4(\mathbb{R}^{n+1}\setminus\{0\})$ and that $\operatorname{Hess}(F^2)$ is positive definite on $\mathbb{R}^{n+1}\setminus\{0\}$ (see Section 2 for details). We now state our main results.

\begin{theorem}\label{theorem1.1}
  Let $\varphi$ be a $C^2$ function in $\mathbb{R}^n$ and let $\Omega\subset\mathbb{R}^n$  be a bounded open set with $C^2$-boundary of non-negative anisotropic mean curvature:
\begin{equation}\label{e1.1}
F_{\xi_i\xi_j}(\nabla d,0)d_{ij}\leq 0,
\end{equation}
where $i$ and $j$  range from $1$ to $n$ under the Einstein summation convention, denote $F_{\xi_i}(\xi)=\frac{\partial F}{\partial\xi_i}$ and $u_i=\frac{\partial u}{\partial x_i}$. Then, the following Dirichlet problem for the anisotropic minimal surface equation:
\begin{equation}\label{e1.2}\begin{cases}
F_{\xi_i\xi_j}(Du,-1)u_{ij}=0&\text{in }\Omega\\
u=\varphi&\text{on }\partial\Omega.
\end{cases}\end{equation}
is uniquely solvable in $C^{0,1}(\overline{\Omega})\cap C^2(\Omega)$.
\end{theorem}

The theorem above is stated under the assumption that $\varphi$ is of class $C^2$. However, by proving an a priori estimate for the gradient, it actually suffices for $\varphi$ to be continuous on $\partial\Omega$.  The estimate is obtained via Moser iteration, following the approach in \cite{S76}.

\begin{theorem}\label{theorem1.2}
 Let $\Omega$ be a bounded domain in $\mathbb{R}^n$ with $C^2$-boundary $\partial\Omega$ of non-negative anisotropic mean curvature, and let $\varphi$ be a continuous function on $\partial\Omega$. Then, the Dirichlet problem for the anisotropic minimal surface equation \eqref{e1.2} has a unique solution $u$ belonging to $C^0(\overline{\Omega})\cap C^2(\Omega)$.
\end{theorem}

\begin{remark}
To prove Theorem \ref{theorem1.2}, we require an a priori gradient estimate. Let $\nu = \frac{(Du,-1)}{\sqrt{1+|Du|^2}}$ be the unit normal to the graph of $u$. For the classical minimal surface equation,  by choosing suitable test functions  one obtains an $L^\infty$-$L^1$ estimate. Then the $L^1$-integral is controlled using the identity
\[\delta_i\delta_i w=|\delta w|^2+c^2,\]
which follows from a direct computation for the standard tangential operators $\delta_i=D_i-\nu_i\nu_h D_h$ and $w=-\log(-\nu_{n+1})$.

  In the anisotropic setting,  we introduce the anisotropic tangential operator $\delta_i^F = F(\nu) D_i - \nu_i F_{\xi_j}(\nu) D_j.$ Such a pointwise estimate cannot be obtained by simple direct calculations   due to the complicated nonlinear coupling of the metric $F$, which  forces us to adopt deeper anisotropic geometric and variational structures, specifically the spectral properties of the anisotropic Jacobi operator.
Since suitable test functions are difficult to find in the anisotropic case, we first derive a local $L^\infty$-$L^1$ estimate via Moser's iteration following the strategy of \cite{S76}. We then control the $L^1$-norm using the the anisotropic Jacobi equation established in \cite{CL}  for the vertical translation Jacobi field $\nu_{n+1}$:
\[\operatorname{div}\left(F_{\xi_i\xi_j}(\nu)(\nabla   \nu_{n+1})_j\right)+|S_F|^2\nu_{n+1}=0,\]
  where $S_F$ is the $F$-anisotropic shape operator on the graph of $u$. This provides a crucial divergence-type  equality for $w $:
\[\operatorname{div}\left(F_{\xi_i\xi_j}(\nu)  w_j\right)=F_{\xi_i\xi_j}(\nu) w_{i}  w_j+|S_F|^2,\]
which together with the ellipticity of $F$, yields the necessary integral control  for the gradient estimate.

\end{remark}

As a consequence of a priori estimate for the gradient, we have a weaker forms of Bernstein's theorem.
\begin{theorem}\label{corollary4.1}
Let $u$ be a solution of the anisotropic minimal surface equation \eqref{e4.15} in $\mathbb{R}^n$. Suppose that for any $x\in\mathbb{R}^n$,
\begin{equation}\label{e4.27}
u(x)\leq C(1+|x|),
\end{equation}
for some constant $C$. Then $u$ is an affine function.
\end{theorem}

\begin{remark}\label{remark4.1}
We emphasize that the one-sided linear growth assumption \eqref{e4.27} is sharp. Indeed, Mooney-Yang \cite{MY} constructed smooth, nonlinear entire anisotropic minimal graphs over $\mathbb{R}^4$ satisfying
\[ \sup_{B_r} u \sim r^{1+\mu} \]
for any $\mu \in (0, 1/2)$. Since $\mu > 0$ can be chosen arbitrarily small, these counterexamples demonstrate that the linear growth condition is the critical threshold for the validity of the anisotropic Bernstein-type theorem.
\end{remark}

The restriction on the anisotropic mean curvature of $\partial\Omega$ can also be avoided by a suitable generalization of the Dirichlet problem. More precisely, we introduce the boundary condition in the functional under consideration as a penalization, and we look for a minimum of \[\mathcal{E}_{F}(v,\Omega)=\int_{\Omega}F(Dv,-1)dx+\int_{\partial\Omega}F(\nu,0)|v-\varphi|d\mathcal{H}^{n-1},\]
where $\nu$ is the normal vector of $\partial\Omega$. A solution of the Dirichlet problem \eqref{e1.2} also minimizes $\mathcal{E}_{F}$. The new functional always has a minimum in $BV(\Omega)$, independently of the anisotropic mean curvature of the boundary.  Moreover, we can prove the following regularity theorem.

\begin{theorem}\label{theorem1.3}
 Let $\Omega$ be a bounded domain in $\mathbb{R}^n$ with $C^1$-boundary $\partial\Omega$, and let $\varphi\in L^1(\partial\Omega)$. Then $\mathcal{E}_{F}(v,\Omega)$ admits a minimizer $u \in BV(\Omega)$. Moreover, $u$ is locally Lipschitz continuous in $\Omega$.
\end{theorem}

\begin{remark}\label{remark5.2}
In general, the minimizer $u \in BV(\Omega)$ obtained via the direct method is only a generalized solution. Without appropriate curvature constraints on $\partial\Omega$, the trace of $u$ may not coincide with $\varphi$. Consequently, the classical Dirichlet condition is relaxed into a boundary complementarity condition
\[    \langle D_{\xi'} F(Du, -1), \nu_{\Omega} \rangle + F(\nu_{\Omega}, 0)\sigma = 0, \quad \sigma \in \mathrm{Sign}(\mathrm{Tr}\,u - \varphi) \quad \text{on } \partial\Omega.\]
Structurally, this complementarity functions as a nonlinear Neumann-type condition.  This observation explicitly illustrates why the Neumann problem does not demand the rigid curvature constraints inherent to the Dirichlet problem.
\end{remark}


Recently, Cui and Yip \cite{CY} investigated the anisotropic graphical mean curvature flow with Dirichlet boundary data in arbitrary dimensions (see Section 4 of their paper for details). Their approach yields boundary gradient estimates under a weighted mean-convexity condition involving certain constants $\gamma_1$ and $\gamma_2$. These constants depend quantitatively on the deviation of the anisotropic integrand $F$ from the isotropic one, as well as on the tangential $C^2$-norm of the Dirichlet boundary data. More specifically, their assumptions require $F$ to be sufficiently close to the isotropic function and the Dirichlet boundary data to not vary too rapidly along $\partial\Omega$. They further discuss the associated elliptic translator problem and the long-time behavior of the flow.

In the present paper, we focus on the stationary zero-speed case, namely the Dirichlet problem for the anisotropic minimal surface equation. While our result recovers and refines the corresponding zero-speed Dirichlet case in \cite{CY}, the main difference lies in the boundary hypothesis. Instead of imposing a weighted curvature condition, we assume that the natural anisotropic mean curvature of the vertical cylinder $\partial\Omega\times\mathbb{R}$ is non-negative in the sense of Wang-Xia \cite{WX}. This assumption depends solely on $F$ and the geometry of $\partial\Omega$, clearly independent of the prescribed boundary data. Hence, within the stationary elliptic regime, our result provides a more intrinsic geometric solvability criterion, entirely removing the smallness requirements on both the boundary data and the deviation of $F$ from the isotropic case. We remark that our comparison with \cite{CY} is restricted to this overlapping static setting; their parabolic estimates, translator problem, and long-time analysis are beyond the scope of the present paper.




\medskip

This paper is organized as follows. In Section 2, we recall the definition of the anisotropic perimeter and derive several useful properties of the anisotropic area functional $\mathcal{A}_F$. In Section 3, we establish an anisotropic weak maximum principle and prove Theorem~\ref{theorem1.1} by constructing suitable barrier functions under the assumption of non-negative anisotropic mean curvature. In Section 4, we establish crucial a priori gradient estimates, which are subsequently employed in Section 5 to complete the proof of Theorem~\ref{theorem1.2} and   Theorem~\ref{corollary4.1}. Finally, in Section 6, we employ the direct method to obtain the existence of anisotropic non-parametric minimal surfaces with prescribed boundary data, as stated in Theorem~\ref{theorem1.3}.

 \section{Anisotropic perimeter}
 Throughout this paper, let $F:\mathbb{R}^{n+1}\to[0,+\infty)$ be a nonnegative convex function of class $C^4(\mathbb{R}^{n+1}\setminus\{0\})$, which is even and positively homogeneous of degree 1, so that
\[F(t\xi)=|t|F(\xi)\quad\text{for any }t\in\mathbb{R}, \xi\in\mathbb{R}^{n+1}.\]
Note that there are positive constants $\alpha$ and $\beta$ such that
\begin{equation}\label{e2.1}
\alpha |\xi|\leq F(\xi)\leq \beta |\xi|\quad\text{for any }\xi\in\mathbb{R}^{n+1}.
\end{equation}

\medskip
We can assume without loss of generality that the convex closed set
\[K=\{x\in\mathbb{R}^{n+1}: F(x)\leq 1\}\]
has  measure $|K|$ equal to the measure $w_{n+1}$ of the unit ball in $\mathbb{R}^{n+1}$. Sometimes, we say that $F$ is the  gauge of $K$. The dual metric of $F$ is
\[F^o(x)=\sup_{\xi\in K}\left<\xi,x\right>.\]
It is easy to verify that $F^o(x):\mathbb{R}^{n+1}\to[0,+\infty)$ is also a convex, positively homogeneous function of degree one. In fact, $F^o$ and $F$ are polar to each other in the sense that:
\[F^o(x)=\sup_{\xi\neq0}\frac{\left<x,\xi\right>}{F(\xi)}\qquad F(\xi)=\sup_{x\neq0}\frac{\left<x,\xi\right>}{F^o(x)}.\]
As a consequence, we have the following Cauchy-Schwarz type inequality:
 \[|\langle x,\xi \rangle|\leq F(x)F^o(\xi).\]
It is clear that $F^o$ is the gauge function of the set
\[K^o=\{x\in\mathbb{R}^{n+1}: F^o(x)\leq 1\}.\]
We say that $K^o$ and $K$ are polar to each other, and we denote the measure of $K^o$ by $\kappa_{n+1}$.  For further details, we refer to the literature \cite{AFT,LA,R}.

\medskip
Let $G\subset\mathbb{R}^{n+1}$ be an open set. The  total variation of a function $u\in BV(G)$ with respect to a gauge function $F$ is given by
\[\int_{G}|Du|_{F}dx=\sup\left\{\int_{G}u\mathrm{div}\sigma dx: \sigma\in C_0^1(G,\mathbb{R}^{n+1}),\  F^o(\sigma)\leq 1\right\}.\]
This yields the "generalized" definition of perimeter of a set $E$ with respect to $F$:
\[P_{F}(E,G)=\int_{G}|D\chi_{E}|_{F}dx=\sup\left\{\int_{E}\mathrm{div}\sigma dx: \sigma\in C_0^1(G,\mathbb{R}^{n+1}),\  F^o(\sigma)\leq 1\right\}.\]
The following co-area formula
\[\int_{G}|Du|_{F}dx=\int_{-\infty}^{+\infty}P_{F}(\{u>s\},G)ds,\quad\forall u\in BV(G),\]
and the identity
\[P_{F}(E,G)=\int_{G\cap\partial^* E}F(\nu)d\mathcal{H}^{n}\]
hold, where $\partial^* E$ is the reduced boundary of $E$ and $\nu$ is the outer normal to $E$ (see \cite{AB}).

\medskip
Let $\Omega$ be a bounded open set in $\mathbb{R}^n$. We extend the anisotropic area functional to functions $u\in BV(\Omega)$:
\begin{equation}\label{e2.2}
\mathcal A_F(u,\Omega)=\sup\left\{\int_{\Omega}(g_{n+1}+u \mathrm{div}g') dx: g=(g', g_{n+1})\in C_0^1(\Omega,\mathbb{R}^{n+1}),\ F^o(g)\leq 1.\right\}
\end{equation}
In particular, if $u\in C^1(\Omega)$, this coincides with the classical expression $\int_{\Omega} F(Du,-1) dx$.

\medskip
Therefore, we can prove that $\mathcal{A}_{F}(u,\Omega)$ is lower semi-continuous with respect to weak $L^1$-convergence.
\begin{lemma}\label{lemma2.1}(Semicontinuity)
Let $\Omega\subset\mathbb{R}^n$ be an open set and let $\{u_j\}$ be a sequence of functions in $BV(\Omega)$ which converge  to   $u$ in $L_{loc}^1(\Omega)$. Then
\begin{equation}\label{e2.3}
\mathcal{A}_{F}(u,\Omega)\leq\liminf_{j\to\infty}\mathcal{A}_{F}(u_j,\Omega),
\end{equation}
where $\mathcal{A}_{F}$ is the anisotropic minimal surface functional defined as in \eqref{e1.5}.
\end{lemma}
\begin{proof}
Let $g=(g',g_{n+1})\in C_0^1(\Omega,\mathbb{R}^{n+1})$ be such that $F^o(g)\leq1$, then
\[ \int_{\Omega}(g_{n+1}+u\mathrm{div}g')dx=\lim_{j\to\infty}\int_{\Omega}(g_{n+1}+u_j\mathrm{div}g')dx \leq \liminf_{j\to\infty}\mathcal{A}_{F}(u_j,\Omega).\]
Taking the supremum over all such $g$, \eqref{e2.3} follows.
\end{proof}

\medskip
Now, we recall some properties of the $1$-homogeneous function $F$ (see \cite{WX}). If we restrict $F$ to $\mathbb{S}^{n}$, then from the assumptions on the function $F$, we see that $F$ is positive and convex on $\mathbb{S}^{n}$ and the restriction of
\[F_{\xi\xi}(\xi)=(F_{\xi_i\xi_j}(\xi))_{i,j=1}^{n+1}\]
to $\xi^{\perp}:=\{V\in\mathbb{R}^{n+1}:\left<V,\xi\right>=0\}$ is a positive definite endomorphism $\xi^{\perp}\to \xi^{\perp}$ for all $\xi\in\mathbb{S}^n$, that is, there exists a positive constant $0<\lambda\leq\Lambda$, such that for any $\xi\in\mathbb{S}^n$, $V\in \xi^{\perp}$, we have
\[\lambda|V|^2\leq F_{\xi_i\xi_j}(\xi)V_iV_j\leq\Lambda|V|^2.\]
Moreover, we introduce the following lemma.
\begin{lemma}(\cite{WX}) The following two statements about $F$ are equivalent:
\begin{itemize}
\item[(1)] $\operatorname{Hess}(F^2)$ is positive definite in $\mathbb{R}^{n+1}\setminus\{0\}$;
\item[(2)] The restriction of
\[F_{\xi\xi}(\xi)=(F_{\xi_i\xi_j}(\xi))_{i,j=1}^{n+1}\]
to $\xi^{\perp}$ is a positive definite endomorphism $\xi^{\perp}\to \xi^{\perp}$ for all $\xi\in\mathbb{S}^{n}$.
\end{itemize}
\end{lemma}

\section{Proof of Theorem \ref{theorem1.1}}

Let $\Omega\subset\mathbb{R}^n$ be a bounded domain, we will always suppose that the boundary of $\Omega$, denoted by $\partial\Omega$ is assumed to be at least Lipschitz-continuous. The boundary datum $\varphi$ is supposed to be Lipschitz-continuous, in general $\varphi\in C^2$ is sufficient. In this section, we shall prove Theorem \ref{theorem1.1} by constructing suitable barrier functions under the anisotropic mean curvature condition \eqref{e1.1}, we also remark that the Dirichlet problem \eqref{e1.2} may not have a solution without \eqref{e1.1}.

\medskip
 The work space is $C^{0,1}(\Omega)$, which is of Lipschitz-continuous function in $\Omega$, i.e., continuous functions with finite Lipschitz constant
\[[u]_{\Omega}=\sup\left\{\frac{|u(x)-u(y)|}{|x-y|}; x,y\in\Omega, x\neq y\right\}.\]

\begin{definition}\label{definition3.1}
For $k>0$, we set
\[L_{k}(\Omega)=\{u\in C^{0,1}(\Omega): [u]_{\Omega}\leq k\}.\]
If $\varphi\in C^{0,1}(\partial\Omega)$, we define
\[ L_k(\Omega,\varphi)=\{u\in L_k(\Omega): u=\varphi\text{ on }\partial\Omega\},\quad\text{and}\quad L(\Omega,\varphi)=\{u\in C^{0,1}(\Omega): u=\varphi\text{ on }\partial\Omega\}.\]
\end{definition}

\medskip
We first prove an existence result for the Dirichlet problem of the anisotropic minimal surface equation \eqref{e1.2} in $L_k(\Omega,\varphi)$.
\begin{proposition}\label{proposition3.1}
Let $\varphi$ be a Lipschitz-continuous function on $\partial\Omega$ and assume that $L_k(\Omega,\varphi)$ is non-empty. Then, the anisotropic minimal surface functional $\mathcal{A}_{F}(u,\Omega)$ achieves its unique minimum in $L_k(\Omega,\varphi)$.
\end{proposition}
\begin{proof}
Let $\{u_j\}$ be a minimizing sequence in $L_k(\Omega,\varphi)$. By the Arzela-Ascoli theorem, there exists a subsequence still denote by $\{u_j\}$, converging uniformly to a function $u\in L_k(\Omega,\varphi)$. Then, the semicontinuity of $\mathcal{A}_F$ (see Lemma \ref{lemma2.1}) implies that $u$ attains the minimum of $\mathcal{A}_{F}$ in $L_k(\Omega,\varphi)$.

For uniqueness,   the map $p\mapsto F(p,-1)$ is strictly convex. Therefore, for any $t\in(0,1)$,
\[\begin{split}
\mathcal{A}_{F}\left(tu+(1-t)v\right)=&\int_{\Omega}F(t Du+(1-t)Dv,-1)dx\\
\leq&t\int_{\Omega}F(Du,-1)dx+(1-t)\int_{\Omega}F(Dv,-1)dx\\
=&t\mathcal{A}_{F}(u)+(1-t)\mathcal{A}_{F}(v),
\end{split}\]
and the  inequality is strict whenever $Du\neq Dv$  on a set of positive measure.  If $u$ and $v$ are both minimizers, the convexity inequality must be an equality, which implies that $Du = Dv$ almost everywhere. Since $u-v$ has zero trace and $\Omega$ is connected, we deduce that $u = v$.
\end{proof}

\medskip
Now, we can prove the existence of a minimum for $\mathcal{A}_{F}$ in $L(\Omega,\varphi)$.
\begin{proposition}\label{proposition3.2}
Let $u$ be the minimum for $\mathcal{A}_{F}$ in $L_k(\Omega,\varphi)$. If $[u]_{\Omega}<k$, then $u$ minimizes $\mathcal{A}_{F}$ in $L(\Omega,\varphi)$.
\end{proposition}
\begin{proof}
For $t\in(0,1)$ and $v\in L(\Omega,\varphi)$, let
\[v_t=u+t(v-u).\]
Then, $v_t=\varphi$ on $\partial\Omega$ and $[v_t]_{\Omega}\leq k$ for $t$ small enough. Since $u$ minimizes $\mathcal{A}_{F}$ in $L_k(\Omega,\varphi)$, we have
\[\mathcal{A}_{F}(u,\Omega)\leq \mathcal{A}_{F}(v_t,\Omega).\]
Moreover, the convexity of $\mathcal{A}_{F}$ yields that
\[\mathcal{A}_{F}(u,\Omega)\leq \mathcal{A}_{F}(v_t,\Omega)\leq (1-t)\mathcal{A}_{F}(u,\Omega)+t\mathcal{A}_{F}(v,\Omega),\]
and hence, $\mathcal{A}_{F}(u,\Omega)\leq \mathcal{A}_{F}(v,\Omega)$ for any $v\in L(\Omega,\varphi)$.\\
\end{proof}

Clearly, to prove the existence of the minimum in $L(\Omega,\varphi)$, it is sufficient to get estimates for the Lipschitz constant of the minimizer $u$ of $\mathcal{A}_{F}$ in $L_k(\Omega,\varphi)$.

\medskip
The main tool we apply is the anisotropic weak maximum principle. To state it, we introduce the definition of supersolution and subsolution.

\begin{definition}\label{definition3.2}
We say that a function $u\in L_k(\Omega)$ is a supersolution (resp. subsolution) for $\mathcal{A}_{F}$ in $L_k(\Omega)$ if for any $v\in L_k(\Omega,u)$ with $v\geq u$ (resp. $v\leq u$), we have $\mathcal{A}_{F}(v,\Omega)\geq\mathcal{A}_{F}(u,\Omega)$.
\end{definition}
 In particular, a minimizer of the anisotropic area $\mathcal{A}_{F}$ is both a super and a subsolution.
\begin{remark}\label{remark3.1}
 If $ u \in C^2(\Omega) $ and
$\operatorname{div} D_{\xi'} F(Du, -1) = F_{\xi_i \xi_j}(Du, -1) u_{ij} \leq 0,$
then, for every admissible $ v \geq u$ with the same trace, convexity gives
\[\mathcal{A}_F(v, \Omega) - \mathcal{A}_F(u, \Omega) \geq \int_\Omega D_{\xi'} F(Du, -1) \cdot D(v - u)  dx = -\int_\Omega \operatorname{div} D_{\xi'} F(Du, -1)(v - u)  dx \geq 0.\]
Thus $ u $ is a variational supersolution.
 Conversely, if $u$ is a variational supersolution and $[u]_{\overline{\Omega}}<k$, then for any $0\leq\eta \in C_c^\infty(\Omega)$, differentiating $\mathcal{A}_F(u+t\eta,\Omega)$ at $t=0$ yields
 \[\int_\Omega D_{\xi'} F(Du, -1) \cdot D\eta \, dx \geq 0.\]
Hence $\operatorname{div} D_{\xi'} F(Du, -1) \leq 0$ in the weak sense. The corresponding statement holds for subsolutions.
\end{remark}

\begin{lemma}\label{lemma3.1}(Anisotropic weak maximum principle)
Let $\overline u$ and $\underline u$ be a supersolution and a subsolution for $\mathcal{A}_{F}$ in $L_k(\Omega)$, respectively. If $\overline u\geq \underline u$ on $\partial\Omega$, then $\overline u\geq \underline u$ in $\overline{\Omega}$.
\end{lemma}
\begin{proof}
Assume by contradiction that
\[\Sigma:=\{x\in\Omega: \overline u<\underline u\}\]
is not empty, and let $v=\max\{\underline u,\overline u\}$.

 Then we have $v\in L_k(\Omega,\overline u)$ and $v\geq\overline u$, thus $\mathcal{A}_{F}(v,\Omega)\geq\mathcal{A}_{F}(\overline u,\Omega)$. Since $v=\underline u$ in $\Sigma$ and $v=\overline u$ in $\Omega\setminus\Sigma$, equivalently, we have
\[\mathcal{A}_{F}(\underline u,\Sigma)\geq\mathcal{A}_{F}(\overline u,\Sigma).\]
On the other hand, if we take $v=\min\{\overline u,\underline u\}$, by a similar argument, we show that
\[\mathcal{A}_{F}(\overline u,\Sigma)\geq\mathcal{A}_{F}(\underline u,\Sigma).\]
Therefore, we have
\[\mathcal{A}_{F}(\underline u,\Sigma)=\mathcal{A}_{F}(\overline u,\Sigma).\]

Since $\overline u<\underline u$ in $\Sigma$ and $\overline u=\underline u$ on $\partial \Sigma$, it follows that $D\overline u\neq D\underline u$ in a set of positive measure. The strict convexity of $\mathcal{A}_{F}$ yields that
\[\mathcal{A}_{F}\left(\frac{\overline u+\underline u}{2},\Sigma\right)<\frac{1}{2}\mathcal{A}_{F}(\overline u,\Sigma)+\frac{1}{2}\mathcal{A}_{F}(\underline u,\Sigma)=\mathcal{A}_{F}(\overline u,\Sigma).\]
However, this is impossible because $\overline u$ is a supersolution in $L_k(\Sigma)$ and thus
\[\mathcal{A}_{F}\left(\frac{\overline u+\underline u}{2},\Sigma\right)\geq \mathcal{A}_{F}(\overline u,\Sigma).\]
\end{proof}

The following result follows from the anisotropic weak maximum principle immediately.
\begin{corollary}\label{corollary3.1}
Let $\overline u$ and $\underline u$ be a supersolution and a subsolution for $\mathcal{A}_{F}$ in $L_k(\Omega)$, respectively. Then
\[\sup_{x\in\Omega}(\underline u(x)-\overline u(x))\leq\sup_{y\in\partial\Omega}(\underline u(y)-\overline u(y)).\]
\end{corollary}
\begin{proof}
It is easy to verify that $\overline u+\alpha$ is also a supersolution for any $\alpha\in\mathbb{R}$, and
\[\underline u(x)\leq \overline{u}(x)+\sup_{y\in\partial\Omega}(\underline{u}(y)-\overline{u}(y)).\]
Then, the result follows from Lemma \ref{lemma3.1} by taking $\alpha=\sup_{y\in\partial\Omega}(\underline{u}(y)-\overline{u}(y))$.
\end{proof}

\medskip
In particular, if $u$ and $v$ minimize the anisotropic area $\mathcal{A}_{F}$ in $L_k(\Omega)$, then corollary \ref{corollary3.1} holds for both $u-v$ and $v-u$, and therefore
\begin{equation}\label{e3.1}
\sup_{\Omega}|u-v|\leq \sup_{\partial\Omega}|u-v|.
\end{equation}

\begin{lemma}\label{lemma3.2}(Reduction to boundary estimates)
Let $u$ minimize the anisotropic area in $L_k(\Omega)$. Then
\begin{equation}\label{e3.2}
[u]_{\Omega}\leq\sup\left\{\frac{|u(x)-u(y)|}{|x-y|}: x\in\Omega, y\in\partial\Omega\right\}.
\end{equation}
\end{lemma}
\begin{proof}
For any $x_1,x_2\in\Omega$ and $x_1\neq x_2$, let $t=x_2-x_1$. Then, the function
\[u_t(x)=u(x+t)\]
minimizes the anisotropic area $\mathcal{A}_{F}$ in $L_k(\Omega_t)$, where
\[\Omega_t=\{x\in\mathbb{R}^n: x+t\in\Omega\}.\]
Clearly, $x_1\in\Omega\cap\Omega_t$ and thus the set $\Omega\cap\Omega_t$ is non-empty. Since $u$ and $u_t$ minimize $\mathcal{A}_{F}$ in $L_k(\Omega\cap\Omega_t)$, it follows from \eqref{e3.1} that there exists some $y\in\partial(\Omega\cap \Omega_t)$ such that
\[|u(x_1)-u(x_2)|=|u(x_1)-u_t(x_1)|\leq|u(y)-u_t(y)|=|u(y)-u(y+t)|.\]
Notice that at least one of  $y,y+t$ belongs to $\partial\Omega$, therefore we have
\[\frac{|u(x_1)-u(x_2)|}{|x_1-x_2|}\leq\sup\left\{\frac{|u(x)-u(y)|}{|x-y|}: x\in\Omega, y\in\partial\Omega\right\},\]
and \eqref{e3.2} holds immediately.
\end{proof}

Combining  proposition \ref{proposition3.2} and lemma \ref{lemma3.2}, we conclude that to prove the existence of a minimum for the anisotropic area $\mathcal{A}_{F}$ in $L_k(\Omega)$, we only need to estimate $|u(x)-u(y)|$ for $x\in\Omega$ and $y\in\partial\Omega$.  For any $x\in\Omega$ and $t>0$, we set
\[\Sigma_t=\{x\in\Omega:d(x)<t\}\quad\text{and}\quad \Gamma_t=\{x\in\Omega:d(x)=t\},\]
where $d(x)$ denotes the distance from $x$ to $\partial\Omega$.

\medskip
Now, we introduce the definition of barrier functions, which play  an important role in the proof of Theorem \ref{theorem1.1}.
\begin{definition}\label{definition3.3}
Let $\varphi$ be a Lipschitz-continuous function on $\partial\Omega$. We say that a Lipschitz-continuous function $v$ defined on $\overline{\Sigma}_{t_0}$ for some $t_0>0$ is an upper barrier, if it satisfies
\begin{equation}\label{e3.3}
v=\varphi\text{ on }\partial\Omega;\ v\geq\sup_{\partial\Omega}\varphi\text{ on }\Gamma_{t_0};
\end{equation}
\begin{equation}\label{e3.4}
 v \text{ is a supersolution in }\Sigma_{t_0}.
\end{equation}
Similarly, we say that a Lipschitz-continuous function $v$ defined in some $\Sigma_{t_0}$ is a lower barrier, if it satisfies
\begin{equation}\label{e3.5}
v=\varphi\text{ on }\partial\Omega;\quad\ v\leq\inf_{\partial\Omega}\varphi\text{ on }\Gamma_{t_0};
\end{equation}
\begin{equation}\label{e3.6}
v \text{ is a subsolution in }L_k(\Sigma_{t_0}).
\end{equation}
\end{definition}

\medskip
Now, we are in a position to prove Theorem \ref{theorem1.1}.
\begin{proof}[Proof of Theorem \ref{theorem1.1}]
We first claim that the anisotropic area $\mathcal{A}_{F}$ attains its minimum in $L(\Omega,\varphi)$ if there exist upper and lower barriers $v_{\pm}$ relative to $\varphi$.

Let $L=[v_{\pm}]_{\Sigma_{t_0}}$ and let $Q<k$. Assume that $u$ give the minimum for $\mathcal{A}_{F}$ in $L_k(\Omega)$, then $u$ minimizes $\mathcal{A}_{F}$ in $L_k(\Sigma_{t_0})$. From the anisotropic weak maximum principle (see Lemma \ref{lemma3.1}), we obtain that for any $x\in\Omega$,
\begin{equation}\label{e3.7}\inf_{\partial\Omega}\varphi\leq u(x)\leq\sup_{\partial\Omega}\varphi.\end{equation}
In particular, $v_-(x)\leq u(x)\leq v_+(x)$ on $\Gamma_{t_0}$. Applying  Lemma \ref{lemma3.1} again, we have
\[v_-(x)\leq u(x)\leq v_+(x)\quad\text{ in }\Sigma_{t_0}.\]
Since $u=v_{\pm}$ on $\partial\Omega$, we get
\begin{equation}\label{e3.8}
|u(x)-u(y)|\leq L|x-y|\quad\text{ for any }x\in\Sigma_{t_0}, y\in\partial\Omega.
\end{equation}
On the other hand, for $x\in\Omega\setminus\Sigma_{t_0}$, it follows from \eqref{e3.7} that
\[|u(x)-u(y)|\leq \max\left\{\sup_{\partial\Omega}\varphi-u(y),\ u(y)-\inf_{\partial\Omega}\varphi\right\}\leq Lt_0\leq L|x-y|,\]
and therefore \eqref{e3.8} holds for any $x\in\Omega$. Applying Lemma \ref{lemma3.2}, we have $[u]_{\Omega}\leq L<k$, and the claim follows from proposition \ref{proposition3.2} immediately.

Now, it remains to establish the existence of the barrier functions. We shall discuss only the case of upper barriers and consider
\[v(x)=\varphi(x)+\psi(d(x)),\]
where $\varphi\in C^2(\mathbb{R}^n)$ and $\psi$ is a $C^2$-function on $[0,d_0]$ satisfying
\[\psi(0)=0, \psi'(t)\geq 1, \psi''(t)<0\]
and
\[\psi(d_0)\geq M=2\sup_{\Omega}|\varphi|,\]
where $d_0$ will be determined later.

In this way, condition \eqref{e3.3} is satisfied in $\Gamma_{d_0}$, we only need to show that the $C^2$ function $v$ is a supersolution in $\Sigma_{d_0}$. From the homogeneity of $F$, we deduce that
\begin{equation}\label{e3.9}\begin{split}
F_{\xi_i\xi_j}(Dv,-1)v_{ij}=&F_{\xi_i\xi_j}(Dv,-1)(\varphi_{ij}+\psi''d_id_j+\psi'd_{ij})\\
=&\frac{1}{\sqrt{1+|Dv|^2}}[\psi''F_{\xi_i\xi_j}(\nu)d_id_j+\psi'F_{\xi_i\xi_j}(\nu)d_{ij}+F_{\xi_i\xi_j}(\nu)\varphi_{ij}],
\end{split}\end{equation}
where $\nu=\frac{(Dv,-1)}{\sqrt{1+|Dv|^2}}$. Since $(\operatorname{Hess} F)(\nu)$ is semi-positive definite and $\varphi\in C^2(\Omega)$, we deduce that there exists some constant $c_1(n)>0$ such that
\begin{equation}\label{e3.10}
|F_{\xi_i\xi_j}(\nu)\varphi_{ij}|\leq c_1(n) F_{\xi_i\xi_i}(\nu).
\end{equation}
Moreover, if we set $D\hat{d}=(Dd,0)$, it follows from the property of the function $F$ that
\[F_{\xi_i\xi_j}(\nu)d_id_j=F_{\xi_i\xi_j}(\nu)\hat{d}_i\hat{d}_j>0.\]
Therefore, there exists $c_2(n)$ small such that
\begin{equation}\label{e3.11}
c_2(n)F_{\xi_i\xi_i}(\nu)\leq (1+|Dv|^2)F_{\xi_i\xi_j}(\nu)d_id_j
\end{equation}
for $|Dv|$ large enough.

Now, we need to estimate the term involving $d_{ij}$ in our barrier construction.
Since $Dv = \psi' Dd + D\phi$ and $|Dd| = 1$, a direct asymptotic expansion of the unit normal vector $\nu = \frac{(Dv, -1)}{\sqrt{1 + |Dv|^2}}$ as $\psi' \to \infty$ yields the asymptotic bound
\[
|\nu - (Dd(x), 0)| \le C (\psi')^{-1}.
\]

For $x \in \Sigma_{d_0}$, let $y \in \partial\Omega$ be the unique projection of $x$ onto $\partial\Omega$, which implies $Dd(x) = Dd(y)$. Since $\partial\Omega \in C^2$, the principal curvatures of the parallel surfaces satisfy the classical relation $\kappa_i(x) = \frac{\kappa_i(y)}{1 - d(x)\kappa_i(y)}$ for $1 \le i \le n-1$. For sufficiently small $d_0 > 0$, this immediately yields the Hessian estimate:
\[|D^2 d(x) - D^2 d(y)| \le \max_{1 \le i \le n-1} \left| \frac{d(x)\kappa_i(y)^2}{1 - d(x)\kappa_i(y)} \right| \le C d(x),\]
where $C$ depends only on $d_0$ and the $C^2$-norm of $\partial\Omega$.

Using the boundary curvature condition \eqref{e1.1} at $y$, the uniform Lipschitz continuity of $D^2 F$  and the explicit estimates above, we obtain
\begin{equation}\label{e3.12}
\begin{aligned}
F_{\xi_i \xi_j}(\nu) d_{ij}(x)
&= F_{\xi_i \xi_j}(Dd(y),0) d_{ij}(y) + F_{\xi_i \xi_j}(Dd(y),0) \big[ d_{ij}(x) - d_{ij}(y) \big] \\
&\quad + \Big[ F_{\xi_i \xi_j}(\nu) - F_{\xi_i \xi_j}(Dd(x),0) \Big] d_{ij}(x) \\
&\leq 0 + C d(x) + C \left| \nu - (Dd(x),0) \right| \cdot |D^2 d(x)| \\
&\leq C \big( d(x) + (\psi')^{-1} \big).
\end{aligned}
\end{equation}
Therefore,
\[\psi' F_{\xi_i \xi_j}(\nu) d_{ij}(x)\leq C(\varepsilon+1)\leq  c_3(1+\varepsilon) F_{\xi_i\xi_i}(\nu). \]


In conclusion, if we take $\psi(d)=\varepsilon\log(1+Kd)$ for $K\geq d_0^{-1}e^{R/\varepsilon}$ with $R$ large, then from \eqref{e3.10}-\eqref{e3.12}, we obtain
\[\begin{split}
&\psi'' F_{\xi_i\xi_j}(\nu)d_id_j+\psi'F_{\xi_i\xi_j}(\nu)d_{ij}+F_{\xi_i\xi_j}(\nu)\phi_{ij}\\
\leq&\psi'' F_{\xi_i\xi_j}(\nu)d_id_j + \big[ c_1 + c_3(1+\varepsilon) \big] c_2^{-1}\left(1+|D\phi+\psi' Dd|^2\right)F_{\xi_i\xi_j}(\nu)d_id_j \\
=&\psi'' F_{\xi_i\xi_j}(\nu)d_id_j\left[1-\varepsilon \big[ c_1 + c_3(1+\varepsilon) \big]c_2^{-1}(\psi')^{-2}(1+|D\phi+\psi' Dd|^2)\right]\\
\leq&\psi'' F_{\xi_i\xi_j}(\nu)d_id_j \left( 1-4\varepsilon \big[ c_1 + c_3(1+\varepsilon) \big] c_2^{-1} \right).
\end{split}\]
By choosing $\varepsilon > 0$ sufficiently small such that $4\varepsilon [c_1 + c_3(1+\varepsilon)] < c_2$, we obtain $F_{\xi_i\xi_j}(Dv,-1)v_{ij} < 0$. Thus, $v$ is a valid upper barrier. An analogous argument establishes the existence of the lower barrier, which completes the proof of Theorem \ref{theorem1.1}.



\end{proof}

As an end of this section, we will show that the condition that the anisotropic mean curvature of $\partial\Omega$ is nowhere negative  is necessary for general solvability of the Dirichlet problem \eqref{e1.2}.

We suppose that $\Omega$ is a connected open set, whose boundary is the union of four disjoint sets:
\[\partial\Omega=\Gamma_+\cup\Gamma_-\cup\Gamma_0\cup N,\]
where $\Gamma_+,\Gamma_-,\Gamma_0$ are  relatively open in $\partial\Omega$, (i.e. there exist disjoint open sets $A_+, A_-$ and $A_0$ such that $\Gamma_+=\partial\Omega\cap A_+$, etc.) and $\mathcal{H}^{n-1}(N)=0$.

\medskip
For $t>0$, we define
\[\Omega_t=\{x\in\Omega:\mathrm{dist}(x,\partial\Omega)>t\}.\]

\begin{lemma}\label{lemma3.3}
Let $\Omega$ be as above, and let $u$ and $v$ be two functions of class $C^2(\Omega)\cap C^0(\Omega\cup\Gamma_0)$ such that
\begin{itemize}
\item[($1$)] $\mathrm{div}(D_{\xi'}F(Dv,-1))\leq \mathrm{div}(D_{\xi'}F(Du,-1))$\quad\text{in }$\Omega$;
\item[($2$)] $u\leq v$ on $\Gamma_0$;
 \item[($3$)] for any open set $V\supset (\Gamma_0\cup\Gamma_-)$,
 \[\lim_{t\to 0^+}\int_{\partial\Omega_t\setminus V}\left(F(\nu,0)-\left<\nu,D_{\xi'}F(Dv,-1)\right>\right)d\mathcal{H}^{n-1}=0;\]
 \item[($4$)]  for any open set $U\supset (\Gamma_0\cup\Gamma_+)$,
 \[\lim_{t\to 0^+}\int_{\partial\Omega_t\setminus U}\left(F(\nu,0)+\left<\nu,D_{\xi'}F(Du,-1)\right>\right)d\mathcal{H}^{n-1}=0.\]
\end{itemize}
Then,
\begin{itemize}
\item[($i$)]if $\Gamma_0\neq\emptyset$, then $u\leq v$ in $\Omega$;
\item[($ii$)]if $\Gamma_0=\emptyset$, then $u=v+const$.
\end{itemize}
\end{lemma}
\begin{proof}
For any non-negative function $\phi$, we have
\begin{equation}\label{e3.13}
\begin{split}
&\int_{\Omega_t}\left<D_{\xi'}F(Dv,-1)-D_{\xi'}F(Du,-1),D\phi\right>dx\\
=&-\int_{\Omega_t}\phi[\mathrm{div}(D_{\xi'}F(Dv,-1))-\mathrm{div}(D_{\xi'}F(Du,-1))]dx\\
&+\int_{\partial\Omega_t}\phi\left<D_{\xi'}F(Dv,-1)-D_{\xi'}F(Du,-1),\nu\right>d\mathcal{H}^{n-1}\\
\geq&\int_{\partial\Omega_t}\phi\left<D_{\xi'}F(Dv,-1)-D_{\xi'}F(Du,-1),\nu\right>d\mathcal{H}^{n-1}.
\end{split}
\end{equation}
In particular, if we take $\phi=\phi_k=\max\{0,\min\{u-v,k\}\}$, then we have
\[\left<D_{\xi'}F(Dv,-1)-D_{\xi'}F(Du,-1),D\phi_k\right>\]
\[=\begin{cases}
\left<D_{\xi'}F(Dv,-1)-D_{\xi'}F(Du,-1),D(u-v)\right>,&\text{if }0<u-v<k;\\
0&\text{otherwise}.
\end{cases}\]
Therefore, $\left<D_{\xi'}F(Dv,-1)-D_{\xi'}F(Du,-1),D\phi_k\right>\leq 0$ in $\Omega$, and by \eqref{e3.13}, we have
\begin{equation}\label{e3.14}\begin{split}
0\geq&\int_{\Omega_t}\left<D_{\xi'}F(Dv,-1)-D_{\xi'}F(Du,-1),D\phi_k\right>dx\\
\geq&\int_{\partial\Omega_t}\phi_k\left<D_{\xi'}F(Dv,-1)-D_{\xi'}F(Du,-1),\nu\right>d\mathcal{H}^{n-1}.
\end{split}\end{equation}

In the case of $\Gamma_0\neq\emptyset$ and $u<v$ on $\Gamma_0$, let
\[A=\{x\in\Omega:u(x)<v(x)\}.\]
Then, we have $\phi_k=0$ in $A$ and the negativity of the integral in \eqref{e3.14} implies that
\begin{equation}\label{e3.15}\begin{split}
&\int_{\partial\Omega_t}\phi_k\left<D_{\xi'}F(Dv,-1)-D_{\xi'}F(Du,-1),\nu\right>d\mathcal{H}^{n-1}\\
\geq
&\int_{\partial\Omega_t\setminus(A\cup A_-)}\phi_k\left[\left<D_{\xi'}F(Dv,-1),\nu\right>-\left<D_{\xi'}F(Du,-1),\nu\right>\right]d\mathcal{H}^{n-1}\\
&+\int_{\partial\Omega_t\setminus(A\cup A_+)}\phi_k\left[-\left<D_{\xi'}F(Du,-1),\nu\right>+\left<D_{\xi'}F(Dv,-1),\nu\right>\right]d\mathcal{H}^{n-1},
\end{split}
\end{equation}
where $A_+$ and $A_-$ are open sets s.t. $\Gamma_{\pm}=\partial\Omega\cap A_{\pm}$. To estimate \eqref{e3.15}, we recall some properties for $F$. Since $F$ is strict convexity in $\mathbb{R}^{n+1}$, for any vector $X,Y\in\mathbb{R}^{n+1}$, it holds:
\[F(X)\geq F(Y)+\left<DF(Y),X-Y\right>=\left<DF(Y),X\right>.\]
If we take $\widetilde{\nu}=(\nu,0)\in\mathbb{R}^{n+1}$, then
\[\left<D_{\xi'}F(Du,-1),\nu\right>=\left<DF(Du,-1),\widetilde{\nu}\right>\leq F(\widetilde{\nu})=F(\nu,0).\]
Similarly, if we consider $-\widetilde{\nu}$, we deduce that
\begin{equation}\label{e3.16}
-F(-\nu,0)\leq \left<D_{\xi'}F(Du,-1),\nu\right>\leq F(\nu,0).
\end{equation}
Combining with \eqref{e3.15} and \eqref{e3.16}, we conclude that
\[\begin{split}
&\int_{\partial\Omega_t}\phi_k\left<D_{\xi'}F(Dv,-1)-D_{\xi'}F(Du,-1),\nu\right>d\mathcal{H}^{n-1}\\
\geq & \int_{\partial\Omega_t\setminus(A\cup A_-)}\phi_k[\left<D_{\xi'}F(Dv,-1),\nu\right>-F(\nu,0)]d\mathcal{H}^{n-1}\\
&-\int_{\partial\Omega_t\setminus(A\cup A_+)}\phi_k[\left<D_{\xi'}F(Du,-1),\nu\right>+F(-\nu,0)]d\mathcal{H}^{n-1}.
\end{split}\]
Passing to the limit as $t\to0^+$, from ($3$), ($4$) and \eqref{e3.14}, we obtain
\[\int_{\Omega_t}\left<D_{\xi'}F(Dv,-1)-D_{\xi'}F(Du,-1),D\phi_k\right>dx=0.\]
Letting $k\to\infty$, we obtain
\[\int_{\Omega_t}\left<D_{\xi'}F(Dv,-1)-D_{\xi'}F(Du,-1),D\phi\right>dx=0,\]
where $\phi=\max\{0,u-v\}$.

The strict convexity of $F$ implies that $D\phi=0$ in $\Omega$. Since $\phi=0$ on $\Gamma_0$, we have $\phi=0$ and thus $u\leq v$ in $\Omega$.

If we only have the weak inequality $u\leq v$ on $\Gamma_0$, we can replace $v$ by $v+\varepsilon$ and the conclusion ($i$) holds by letting $\varepsilon\to0^+$.

Finally, in the case of $\Gamma_0=\emptyset$, we set $\alpha=v(x_0)-u(x_0)$ for some $x_0\in\Omega$, and replacing $u+\alpha$ by $u$. By a similar argument with $A=\emptyset$, we conclude  that $\phi=0$ and therefore $v=u+\alpha$.
\end{proof}

Now, we can show that if the anisotropic mean curvature of $\partial\Omega$ is negative at some point $x_0\in\partial\Omega$, then there will exist smooth functions $\varphi$ such that the anisotropic area $\mathcal{A}_{F}$ has no minimum in $L(\Omega,\varphi)$.

\begin{remark}
Let now $x_0\in\partial\Omega$ and letting $\delta(x)=dist(x,B_{R}(x_0))=|x-x_0|-R$ in $\Omega\setminus B_R(x_0)$ and $v=\alpha_1+\psi(\delta)$ for some constant $\alpha_1$. We get
\[F_{\xi_i\xi_j}(Dv,-1)v_{ij}=\psi''F_{\xi_i\xi_j}(\psi' D\delta,-1)\delta_i\delta_j+\psi'F_{\xi_i\xi_j}(\psi' D\delta,-1)\delta_{ij}.\]
Taking $\psi(\delta)=-\beta_1\delta^{1/2}$ and using the homogeneity of $F$,  we get
\[\begin{split}
\psi'' F_{\xi_i\xi_j}(\psi'D\delta,-1)\delta_i\delta_j=&\frac{\beta_1}{4}\delta^{-\frac{3}{2}}F_{\xi_i\xi_j}\left(-\frac{\beta_1}{2}\delta^{-\frac{1}{2}}D\delta,-1\right)\delta_i\delta_j\\
=&\frac{\beta_1}{4}\delta^{-\frac{3}{2}}\left|-\frac{\beta_1}{2}\delta^{-\frac{1}{2}}\right|^{-1}F_{\xi_i\xi_j}\left(D\delta,\frac{2\delta^{1/2}}{\beta_1}\right)\delta_i\delta_j\\
=&\frac{2}{\beta_1^2}F_{\xi_{n+1}\xi_{n+1}}\left(D\delta,\frac{2\delta^{1/2}}{\beta_1}\right)\leq\frac{2C_1}{\beta_1^2},
\end{split}\]
where  $C_1$ is a constant independent of $\beta_1$. Moreover, due to the ellipticity of $F$,
\[\begin{split}
\psi' F_{\xi_i\xi_j}(\psi'D\delta,-1)\delta_{ij}=&-\frac{\beta_1}{2}\delta^{-\frac{1}{2}}F_{\xi_i\xi_j}\left(-\frac{\beta_1}{2}\delta^{-\frac{1}{2}}D\delta,-1\right)\delta_{ij}\\
=&-\frac{\beta_1}{2}\delta^{-\frac{1}{2}}\left|-\frac{\beta_1}{2}\delta^{-\frac{1}{2}}\right|^{-1}F_{\xi_i\xi_j}\left(D\delta,\frac{2\delta^{1/2}}{\beta_1}\right)\delta_{ij}\\
=&-\frac{1}{R+\delta}F_{\xi_i\xi_j}\left(D\delta,\frac{2\delta^{1/2}}{\beta_1}\right)(I_{ij}-\delta_i\delta_j)\\
\leq&-\frac{c_0}{R+\delta}\leq-\frac{c_0}{R+\mathrm{diam}(\Omega)},
\end{split}\]
where $I=(I_{ij})$   is the identity matrix. Therefore, we conclude that
\[F_{\xi_i\xi_j}(Dv,-1)v_{ij}\leq0\quad\text{in }\Omega\setminus\overline{B}_{R}(x_0)\]
by choosing $\beta_1^2\geq\frac{2C_1(R+\mathrm{diam}(\Omega))}{c_0}$.

Let now $u$ be the minimum of the anisotropic area $\mathcal{A}_{F}$ in $C^{0,1}(\Omega)$. If we choose $\alpha_1=\sup_{\partial\Omega\setminus B_{R}(x_0)}u+\beta_1(\mathrm{diam}(\Omega))^{1/2}$, then we have $u\leq v$ on $\partial\Omega\setminus B_{R}(x_0)$. Since $D_{\xi'}F(Dv,-1)=D_{\xi'}F(D\delta,2\delta^{1/2}/\beta_1)$ and $D\delta=\frac{x-x_0}{R}$ is the unit normal vector on $\partial B_{R}(x_0)$, we deduce that
\[\lim_{\delta\to0}\left<D_{\xi'}F(Dv,-1),\nu\right>=\left<D_{\xi'}F(D\delta,0),D\delta\right>=F(D\delta,0)=F(\nu,0)\quad\text{on }\partial B_{R}(x_0).\]
Therefore, we can apply Lemma \ref{lemma3.3} with $\Gamma_+=\emptyset$ in $\Omega\setminus\overline{B}_{R}(x_0)$ to conclude that
\[\sup_{\Omega\setminus B_R(x_0)}u\leq v\leq \sup_{\partial\Omega\setminus B_R(x_0)}u+\beta_1\mathrm{diam}(\Omega)^{1/2}.\]
In particular,
\begin{equation}\label{e3.17}
\sup_{\Omega\cap\partial B_R(x_0)}u\leq v\leq \sup_{\partial\Omega\setminus B_R(x_0)}\varphi+\beta_1 \mathrm{diam}(\Omega)^{1/2}.
\end{equation}

Suppose now  the anisotropic mean curvature of $\partial\Omega$ at $x_0$ is negative, and let $R$ be small such that
\[F_{\xi_i\xi_j}(Dd,0)d_{ij}\geq\varepsilon_0>0\quad\text{in }\Omega\cap B_{R}(x_0),\]
where $d(x)$ is the distance from $x$ to $\partial\Omega$. Similarly, we take $v=\alpha_2-\beta_2d^{1/2}$ and deduce that
\[
F_{\xi_i\xi_j}(Dv,-1)v_{ij}= \frac{2}{\beta_2^2}F_{\xi_{n+1}\xi_{n+1}}\left(Dd,\frac{2d^{1/2}}{\beta_2}\right)-F_{\xi_i\xi_j}\left(Dd,\frac{2d^{1/2}}{\beta_2}\right)d_{ij}
\leq \frac{2C_1}{\beta_2^2}-\frac{\varepsilon_0}{2}\leq0
 \]
for $\beta_2^2\geq\frac{4C_1}{\varepsilon_0}$.

If we take $\alpha_2=\sup_{\Omega\cap\partial B_R(x_0)}u+\beta_2\operatorname{diam}(\Omega)^{1/2}$, using again the Lemma \ref{lemma3.3} in $ {\Omega}\cap B_{R}(x_0)$, we have
\[\sup_{\overline{\Omega}\cap B_R(x_0)}u\leq v\leq \sup_{\partial B_{R}(x_0)\cap\Omega}u+\beta_2\operatorname{diam }(\Omega)^{1/2}.\]
In particular,
\begin{equation}\label{e3.18}
\sup_{\partial\Omega\cap B_R(x_0)}\varphi\leq\sup_{\partial B_R(x_0)\cap\Omega}u+\beta_2\operatorname{diam }(\Omega)^{1/2}.
\end{equation}
Combining with \eqref{e3.17} and \eqref{e3.18}, we conclude that
\begin{equation}\label{e3.19}
\sup_{\partial\Omega\cap B_R(x_0)}\varphi\leq\sup_{\partial\Omega\setminus B_R(x_0)}\varphi+(\beta_1+\beta_2)\operatorname{diam }(\Omega)^{1/2}.
\end{equation}
Therefore, if we take a smooth function $\varphi$ such that \eqref{e3.19} is not satisfied, then the Dirichlet problem for the anisotropic area $\mathcal{A}_{F}$ with boundary datum $\varphi$ cannot have a solution.

For example, choose $\eta\in C^\infty(\partial\Omega)$ such that
$0\le \eta\le 1$, $\eta\equiv 1$ on
$\partial\Omega\cap B_{R/2}(x_0)$, and $\eta\equiv 0$ on
$\partial\Omega\setminus B_R(x_0)$. Setting
\[
    \varphi=M\eta ,
\]
we see that \eqref{e3.19} fails for $M$ sufficiently large. Hence the
Dirichlet problem with this smooth boundary trace cannot have a solution.
\end{remark}

\section{A priori estimate of the gradient}
In this section, our main purpose is to establish a priori estimate for  the gradient of $u$ in terms of the supremum of $u$.

\medskip
We first introduce some notations. Let $ {B}_R$ denote the ball centered at the origin in $\mathbb{R}^{n}$ with radius $R$
 and let   $u$ be a solution of \eqref{e1.6} in $B_R$.
We shall denote its graph by $S$ and let the normal vector to $S$ at the point $(x,u(x))$ be $\nu=\frac{(Du,-1)}{\sqrt{1+|Du|^2}}.$
We also set
\begin{equation}\label{e4.3}v=\sqrt{1+|Du|^2},\quad\quad \nu_i=u_i/v\text{ for }i=1,\cdots,n.\end{equation}
and \begin{equation}\label{e4.4}g^{ij}=\delta_{ij}-\nu_i\nu_j,\quad i,j=1,\cdots, n.\end{equation}
For any $X\in\mathbb{R}^n$, we have $g^{ij}X_iX_j=|X|^2-\langle X,\overline{\nu}\rangle^2$, where $\overline{\nu}=Du/\sqrt{1+|Du|^2}$ are the first $n$ components of $\nu$. Let $\widetilde{X}=(X,0)$ and let  $Y=\widetilde{X}-\langle \widetilde{X},\nu\rangle\nu$ denote  the projection onto the tangent space $\nu^{\bot}$. The  ellipticity of $F$ indicates that there exists some constants $\lambda$ and $\Lambda$ such that
\[\lambda|Y|^2\leq F_{\xi_k\xi_l}(\nu)Y_kY_l\leq\Lambda|Y|^2,\quad 1\leq k,l\leq n+1.\]
Notice that \[|Y|^2=|\widetilde{X}|^2-\langle \widetilde{X},\nu\rangle^2=|X|^2-\langle X,\overline{\nu}\rangle^2=g^{ij}X_jX_j,\quad 1\leq i,j\leq n,\]
and
\[F_{\xi_k\xi_l}(\nu)Y_kY_l= F_{\xi_k\xi_l}(\nu)\left(\widetilde{X}_k-\langle \widetilde{X},\nu\rangle\nu_k\right)\left(\widetilde{X}_l-\langle \widetilde{X},\nu\rangle\nu_l\right)=F_{\xi_i\xi_j}(\nu)X_i X_j, \quad 1\leq i,j\leq n,\]
where we used the homogeneity of $F$. Therefore, for any $X\in\mathbb{R}^n$, we have
\begin{equation}\label{e4.1}
\lambda g^{ij}X_iX_j\leq F_{\xi_i\xi_j}(\nu)X_i X_j\leq\Lambda g^{ij}X_iX_j \quad\text{with }1\leq i,j\leq n.
\end{equation}

\medskip
In the case of non-parametric hypersurfaces in $\mathbb{R}^{n+1}$ of the form $x_{n+1}=u(x)$, where $u(x)$ is a $C^2$ function defined on some bounded domain $\Omega\subset\mathbb{R}^n$, a general Sobolev inequality established in \cite{MS} implies that
\begin{equation}\label{e4.2}
\left\{\int_{\Omega}f^{\frac{n}{n-1}}vdx\right\}^{\frac{n-1}{n}}\leq c(n)\int_{\Omega}\left(\sqrt{g^{ij}f_if_j}+f|H|\right)vdx,
\end{equation}
where $v$ and $g^{ij}$ are defined by \eqref{e4.3} and \eqref{e4.4}, respectively. In fact, the quantity $H=\frac{1}{n}v^{-1}g^{ij}u_{ij}$ is the mean curvature of the hypersurface $x_{n+1}=u(x)$. We denote
\[N=\begin{cases}
\text{a fixed number in }(2,4)&\text{if }n=2\\
n&\text{if }n\geq3.
\end{cases}\]

\medskip

\begin{lemma}\label{lemma4.1}
Suppose that $u$ is a solution of
\begin{equation}\label{e4.15}
\mathrm{div}\left(D_{\xi'}F(Du,-1)\right)=0
\end{equation}
in $B_{\rho}(x_0)$ and set $w=\log\sqrt{1+|Du|^2}$.  Then
 \begin{equation}\label{e4.16}
  \sup_{S_{\rho/2}(x_0)}w(x)\leq c_1 \rho^{-n}\int_{S_\rho(x_0)}wd\mu_g, 
 \end{equation}
 where $S_\rho(x_0)=S\cap\left({B}_{\rho}(x_0)\times\{|u-u(x_0)|<\rho\}\right).$
\end{lemma}

\begin{proof}
Replacing $f$ by $f^{\frac{2N(n-1)}{n(N-2)}}$ in \eqref{e4.2}, we obtain
\[\begin{split}
\left(\int_{S}f^{\frac{2 N}{ N -2}}d\mu_g\right)^{\frac{n-1}{n}}&\leq C(n)\int_{S}\left(\frac{2N(n-1)}{n(N-2)}f^{\frac{2N(n-1)}{n(N-2)}-1}\sqrt{g^{ij}f_if_j}+f^{\frac{2N(n-1)}{n(N-2)}}|H|\right)d\mu_g\\
&\leq C(n)\left(\int_{S}f^{\frac{4N(n-1)}{n(N-2)}-2}d\mu_g\right)^{1/2}\times\left[\int_{S}\left(g^{ij}f_if_j+f^2H^2\right) d\mu_g\right]^{1/2}.
\end{split}\]
When $n>2$, we notice that $\frac{4N(n-1)}{n(N-2)}-2=\frac{2n}{n-2}$ with $n=N$, and the above inequality becomes
\[\left(\int_{S}f^{\frac{2n}{ n-2}}d\mu_g\right)^{\frac{n-2}{n}}\leq C(n) \int_{S}\left(g^{ij}f_if_j+f^2H^2\right) d\mu_g.\]
When $n=2$, we have the relation \[\frac{N-2}{4}=\frac{\theta}{2}+\frac{(N-2)(1-\theta)}{2N}\] where $\theta=\frac{(N-2)^2}{4}$. Then, by the interpolation inequality, we obtain
\[\int_{S}f^{\frac{4N(n-1)}{n(N-2)}-2}d\mu_g=  \int_{S}f^{\frac{4}{N-2}}d\mu_g\leq \left(\int_{S}f^2d\mu_g\right)^{\frac{N-2}{2}}\left(\int_{S}f^{\frac{2N}{N-2}}d\mu_g\right)^{\frac{4-N}{2}},\]
which implies that
\[\left( \int_{S} f^{\frac{2N}{N-2}} \, d\mu_g \right)^{\frac{N-2}{N}} \leq C(n) \left( \int_{S} f^2 \, d\mu_g \right)^{1-\frac{2}{N}} \left[ \int_{S} \left( g^{ij} f_i f_j + f^2 H^2 \right) \, d\mu_g \right]^{\frac{2}{N}}.\]
Summarizing the two cases above, we have
\begin{equation}\label{e4.5}
\left(\int_{S}f^{\frac{2 N}{ N -2}}d\mu_g\right)^{\frac{N-2}{N}}\leq C(n)\left(\int_{S}f^{2}d\mu_g\right)^{1-\frac{n}{N}}\left[\int_{S}\left(g^{ij}f_if_j+f^2H^2\right)d\mu_{g}\right]^{\frac{n}{N}}
\end{equation}

Recall that $u$ satisfies the anisotropic minimal surface equation
\[\mathrm{div}\left(D_{\xi'}F(Du,-1)\right)=0.\]
Multiplying the equation by $\eta\in C_0(\Omega)$ and integrating by parts yields
\[\int_{\Omega}F_{\xi_i}(Du,-1)\eta_i dx=0,\]
where $\eta_i=\partial\eta/\partial x_i$. Taking $\eta\in  C_0^2(\Omega)$ and replacing $\eta$ with $\eta_l$, then integrating by parts, we obtain
\[
\int_{\Omega}F_{\xi_i\xi_j}(Du,-1)u_{jl}\eta_i dx=0.
\]
Finally, replacing $\eta$ in the above identity by $\nu_l\eta$,  we obtain
\begin{equation}\label{e4.6}\begin{split}
0=&\int_{\Omega}F_{\xi_i\xi_j}(Du,-1)u_{jl}(\nu_l\eta_i+v^{-1}g^{lk}u_{ki}\eta)dx\\
=&\int_{\Omega}\left[v^{-1}F_{\xi_i\xi_j}(Du,-1)g^{lk}u_{jl}u_{ki}\eta+F_{\xi_i\xi_j}(Du,-1)v_j\eta_i\right]dx,
\end{split}\end{equation}
where $v=\sqrt{1+|Du|^2}$ and $i,j,k,l=1\cdots n$.

Set $w=\log v$, we can rewrite \eqref{e4.6} as
\[\int_\Omega\left[v^{-2}F_{\xi_i\xi_j}(\nu)g^{lk}u_{jl}u_{ki}\eta+F_{\xi_i\xi_j}(\nu)w_j\eta_i\right]dx=0.\]
Replacing $\eta$ by $v\eta$ and using the identity  $(v\eta)_i=v_i\eta+v\eta_i=vw_i\eta+v\eta_i,$ we obtain
\[F_{\xi_i\xi_j}(\nu)w_j(v\eta)_i\,dx=F_{\xi_i\xi_j}(\nu)w_jw_i\eta\,d\mu_g+F_{\xi_i\xi_j}(\nu)w_j\eta_i\,d\mu_g.\]
Therefore
\begin{equation}\label{e4.29}\int_S\left[F_{\xi_i\xi_j}(\nu)w_j\eta_i+\left(v^{-2}F_{\xi_i\xi_j}(\nu)g^{lk}u_{jl}u_{ki}+F_{\xi_i\xi_j}(\nu)w_jw_i\right)\eta\right]d\mu_g=0.\end{equation}
Replacing $\eta$ by $w^q\eta^2$ with $ q\ge 2$ in \eqref{e4.29}, we get
\begin{equation}\label{e4.30}\begin{split}
&\int_S\left(v^{-2}F_{\xi_i\xi_j}(\nu)g^{lk}u_{jl}u_{ki}+F_{\xi_i\xi_j}(\nu)w_jw_i\right)w^q\eta^2\,d\mu_g+q\int_S F_{\xi_i\xi_j}(\nu)w_jw_iw^{q-1}\eta^2\,d\mu_g\\
&=-2\int_S F_{\xi_i\xi_j}(\nu)w_j\eta_iw^q\eta\,d\mu_g.\end{split}\end{equation}
Young's inequality implies that
\[2|F_{\xi_i\xi_j}(\nu)w_j\eta_i|w^q\eta\le{1\over2}F_{\xi_i\xi_j}(\nu)w_jw_iw^q\eta^2+2w^qF_{\xi_i\xi_j}(\nu)\eta_i\eta_j.\]
Therefore,
\[\begin{split}
&\int_S\left(v^{-2}F_{\xi_i\xi_j}(\nu)g^{lk}u_{jl}u_{ki}+\frac{1}{2}F_{\xi_i\xi_j}(\nu)w_jw_i\right)w^q\eta^2\,d\mu_g+q\int_S F_{\xi_i\xi_j}(\nu)w_jw_iw^{q-1}\eta^2\,d\mu_g\\
&\le 2\int_S w^qF_{\xi_i\xi_j}(\nu)\eta_i\eta_j\,d\mu_g.\end{split}\]
In particular,
\begin{equation}\label{e4.31}\int_S v^{-2}F_{\xi_i\xi_j}(\nu)g^{lk}u_{jl}u_{ki} w^q\eta^2\,d\mu_g \le c\Lambda\int_Sw^q|D\eta|^2\,d\mu_g.\end{equation}
Replace $\eta$ by $ w^{q-1}\eta^2 $ with $q\geq 2$ in \eqref{e4.29},  we obtain
\[\begin{split}
&\int_S F_{\xi_i\xi_j}(\nu)w_j[(q-1)w^{q-2}w_i\eta^2+2w^{q-1}\eta\eta_i] d\mu_g   \\
&+ \int_S\left(v^{-2}F_{\xi_i\xi_j}(\nu)g^{lk}u_{jl}u_{ki}+F_{\xi_i\xi_j}(\nu)w_jw_i\right)w^{q-1}\eta^2 d\mu_g=0,
\end{split}\]
which implies that
\[\begin{split}
(q-1)\int_S F_{\xi_i\xi_j}(\nu)w_jw_i w^{q-2}\eta^2 d\mu_g &\leq2\int_S |F_{\xi_i\xi_j}(\nu)w_j\eta_i| w^{q-1}\eta d\mu_g \\
&\leq 2 \left(\int_S F_{\xi_i\xi_j}(\nu)w_jw_iw^{q-2}\eta^2\,d\mu_g\right)^{1/2}\left(\int_Sw^qF_{\xi_i\xi_j}(\nu)\eta_i\eta_j\,d\mu_g\right)^{1/2}
\end{split}\]
where we used the Cauchy-Schwartz inequality.  Therefore,
\begin{equation}\label{e4.32}\int_SF_{\xi_i\xi_j}(\nu)w_iw_jw^{q-2}\eta^2\,d\mu_g\le{4\over(q-1)^2}\int_Sw^qF_{\xi_i\xi_j}(\nu)\eta_i\eta_j\,d\mu_g\leq 4\Lambda\int_Sw^q|D\eta|^2\,d\mu_g.\end{equation}

Now, we apply \eqref{e4.5} with $f^2=\eta^{Nq-N+2}w^q$. By a direct calculation, we deduce that
\[ \begin{split}
f_if_j= &\left(\frac{Nq-N}{2}+1\right)^2\eta^{Nq-N}w^q\eta_i\eta_j+\frac{q^2}{4}\eta^{Nq-N+2}w^{q-2}w_iw_j \\
&+q\left(\frac{Nq-N}{2}+1\right)\eta^{Nq-N+1}w^{q-1}w_i\eta_j
 \end{split}\]
and
\[f^2 H^2\leq \eta^{Nq-N+2}w^q n^{-1}v^{-2}g^{ij}g^{lk}u_{il}u_{jk}\leq \lambda^{-1} \eta^{Nq-N+2}w^qv^{-2} F_{\xi_i\xi_j}(\nu)g^{lk}u_{jl}u_{ik},\]
where we used \eqref{e4.1} that $g^{ij}X_iX_j\leq \lambda^{-1}F_{\xi_i\xi_j}(\nu)X_iX_j$ in the last inequality. Therefore, we have
\begin{equation}\label{e4.34}\begin{split}
&\left[\int_{S}\left(\eta^{Nq-N+2}w^q\right)^{\frac{N}{N-2}}d\mu_g\right]^{\frac{N-2}{N}}\\
\leq &C(n)\left(\int_{S}\eta^{Nq-N+2}w^qd\mu_g\right)^{1-\frac{n}{N}}\left(\int_{S}q^2\eta^{Nq-N}w^qg^{ij}\eta_i\eta_jd\mu_g
\right.\\
&\left. +\int_{S}q^2\eta^{Nq-N+2}w^{q-2}g^{ij}w_iw_jd\mu_g+ \lambda^{-1}\int_{S}\eta^{Nq-N+2}w^q v^{-2} F_{\xi_i\xi_j}(\nu)g^{lk}u_{jl}u_{ik}d\mu_g\right)^{\frac{n}{N}}\\
 \leq & C(n)\left(\int_{S}\eta^{Nq-N+2}w^qd\mu_g\right)^{1-\frac{n}{N}}\left[q^2\Lambda\lambda^{-1}\int_{S}\eta^{Nq-N}w^q|D\eta|^2d\mu_g
\right.\\
&\left. +q^2 \lambda^{-1}\int_{S}\eta^{Nq-N+2}\left(w^{q-2}F_{\xi_i\xi_j}(\nu)w_iw_j+w^{q}v^{-2} F_{\xi_i\xi_j}(\nu)g^{lk}u_{jl}u_{ik}\right)d\mu_g\right]^{\frac{n}{N}}\\
\end{split}\end{equation}
\[\begin{split}\leq& C(n)\left(\int_{S}\eta^{Nq-N+2}w^qd\mu_g\right)^{1-\frac{n}{N}} \left[q^2\Lambda\lambda^{-1}\int_{S}\eta^{Nq-N}w^q|D\eta|^2d\mu_g\right.\\
&\left.
+cq^2\Lambda\lambda^{-1}(Nq-N+2)^2\int_{S}\eta^{Nq-N}w^{q}|D\eta|^2d\mu_g\right]^{\frac{n}{N}}\\
\leq& C(n,\lambda,\Lambda)\left(\int_{S}\eta^{Nq-N+2}w^qd\mu_g\right)^{1-\frac{n}{N}} \left(q^4\int_{S}\eta^{Nq-N}w^{q}|D\eta|^2d\mu_g\right)^{\frac{n}{N}},
\end{split}\]
where we used \eqref{e4.31} and \eqref{e4.32} with $\eta^2$ replaced by  $\eta^{Nq-N+2}$ in the  third inequality.

Let $\eta_1, \eta_2 \in C_0^\infty(\mathbb{R})$ be standard cut-off functions satisfying:
 \[0 \le \eta_i(s) \le 1, \quad \eta_i(s) = 1 \text{ for } |s| \le \frac{1}{2}, \quad \eta_i(s) = 0 \text{ for } |s| \ge 1, \quad |\eta_i'| \le 4,  \quad i=1,2.\]
We define the product-type cut-off function $\eta: \mathbb{R}^{n+1} \to [0, 1]$ by
\begin{equation}\label{cutoff-def}
\eta(x,x_{n+1}) := \eta_1 \left( \frac{|x - x_0|}{\rho} \right) \eta_2 \left( \frac{|x_{n+1} - u(x_0)|}{\rho} \right).
\end{equation}
Clearly, $\eta(x,x_{n+1}) = 1$ on $S\cap(B_{\rho/2}\times \{|x_{n+1}-u(x_0)|<\rho/2\})$ and $\operatorname{supp}(\eta) \subset S\cap(B_{\rho }\times \{|x_{n+1}-u(x_0)|<\rho \})$.
By computing the gradient of $\eta$ in $\mathbb{R}^{n+1}$, we easily obtain
\[
\left| D_{\mathbb{R}^{n+1}} \eta(X) \right|^2 \le \frac{C}{\rho^2} \left[ \left(\eta_1'\right)^2 \eta_2^2 + \eta_1^2 \left(\eta_2'\right)^2 \right] \le \frac{C}{\rho^2}.
\]
When restricted to the graph $S$, we obtain
\begin{equation}\label{e4.33}
|\nabla_S \eta|^2 \leq \left| D_{\mathbb{R}^{n+1}} \eta \right|^2 \le \frac{C}{\rho^2},
\end{equation}
which holds independently of any prior bound on $Du$.


 Therefore, from \eqref{e4.34} and \eqref{e4.33}, we derive that
  \begin{equation}\label{e4.9}
 \left[\int_{S_\rho(x_0)}\left(\eta^{Nq-N+2}w^q\right)^{\frac{N}{N-2}}d\mu_g\right]^{\frac{N-2}{N}}
\leq C q^4\rho^{-\frac{2n}{N}}\int_{S_\rho(x_0)}\eta^{Nq-N}w^q d\mu_g,
 \end{equation}
where $C$ is a constant depending only on $N,\lambda,\Lambda$. Multiplying both sides of the above inequality by $\rho^{-n/\kappa}$ with $\kappa=N/(N-2)$  and letting
\[I_{q}=\rho^{-n}\int_{S_{\rho}(x_0)}\eta^{Nq-N}w^qd\mu_g=\rho^{-n}\int_{S_{\rho}(x_0)}\eta^{-N}(\eta^{N}w)^qd\mu_g\quad\text{with }q\geq2,\]
we obtain from \eqref{e4.9}  that
\begin{equation}\label{e4.10}
I_{q\kappa}^{1/\kappa}\leq Cq^4 I_{q}.
\end{equation}
Setting $q=2\kappa^{\gamma}$, then \eqref{e4.10} yields
\[\left(I_{2\kappa^{\gamma+1}}\right)^{1/(2\kappa^{\gamma+1})}\leq C^{1/(2\kappa^{\gamma})}(2\kappa^{\gamma})^{2/(\kappa^{\gamma})}\left(I_{2\kappa^{\gamma}}\right)^{1/(2\kappa^{\gamma})}\quad\text{with }\gamma=0, 1,\cdots.\]
Integrating this inequality and note that $\sum_{\gamma}\kappa^{-\gamma}, \sum_{\gamma}\gamma\kappa^{-\gamma}<\infty$, we then have
\begin{equation}\label{e4.11}
 \sup\{\eta^{N}w\}=\lim_{\gamma\to\infty}\left(I_{2\kappa^\gamma}\right)^{1/(2\kappa^\gamma)} \le C^{\frac{1}{2}\sum_{\gamma=0}^\infty \kappa^{-\gamma}} 4^{\sum_{\gamma=0}^\infty \kappa^{-\gamma}} \kappa^{2\sum_{\gamma=0}^\infty \gamma\kappa^{-\gamma}} I_2^{1/2} \le C I_2^{1/2}.
\end{equation}
To reduce the $L^2$-type norm $I_2$ to the $L^1$ norm, we observe that
\[I_2 = \rho^{-n}\int_{S_\rho} \eta^N w^2 d\mu_g \leq  \sup_{S_\rho}\left(\eta^N w\right)  \left( \rho^{-n}\int_{S_\rho} w d\mu_g \right) \leq C (I_2)^{1/2} I_1.\]
This immediately implies $(I_2)^{1/2} \leq C I_1$. Combining this with \eqref{e4.11}, we conclude:
\[
\sup_{S_{\rho/2}(x_0)} w \leq C I_1 = c_1 \rho^{-n} \int_{S_\rho(x_0)} w d\mu_g,
\]
which completes the proof.
\end{proof}

\begin{lemma}\label{lemma4.2}
Let $u$ be a solution of the anisotropic minimal surface equation \eqref{e4.15} in $B_{3\rho}$ with $u(0)=0$ . Then
\begin{equation}\label{e4.14}
\rho^{-n}\int_{S_{\rho}}wd\mu_g\leq c_2\left(1+\rho^{-1}\sup_{B_{3\rho} }u\right),
\end{equation}
 where $S_\rho=S\cap(B_\rho\times(-\rho,\rho) )$.
\end{lemma}
\begin{proof}
Let $\zeta\in C_0^{\infty}(B_{2\rho})$, $0\leq\zeta\leq1$, $\zeta=1$ in $B_{\rho}$ and $|D\zeta|\leq 2\rho^{-1}$. We consider the operator $\delta_{n+1}^F=F(\nu)D_{n+1}-\nu_{n+1}F_{\xi_i}(\nu)D_i$ where $\nu_{n+1}=-v^{-1}$ is the $(n+1)$-th component of the unit normal $\nu$ and $i=1,\cdots,n$. Define
\[u_{\rho}=\begin{cases}
2\rho&\text{if }u\geq \rho\\
u+\rho&\text{if }|u|<\rho\\
0&\text{if } u<-\rho,
\end{cases}\]
then we have
\[\delta_{n+1}^Fu_{\rho}=\begin{cases}
0&\text{if }|u|>\rho\\
F(\nu)-F_{\xi_{n+1}}(Du,-1)\nu_{n+1}&\text{if }|u|<\rho,
\end{cases}\]
where we used the relation that $F_{\xi_i}(\nu)u_i=F(Du,-1)+F_{\xi_{n+1}}(Du,-1)$.

Let $C_{2\rho}=S\cap(B_{2\rho}\times\mathbb{R})$. The divergence theorem implies that
\[\begin{split}\int_{C_{2\rho}}\delta_{n+1}^F(u_{\rho}\zeta w) d\mu_g
=&-\int_{C_{2\rho}}\nu_{n+1}F_{\xi_i}(\nu) D_i(u_{\rho}\zeta w) d\mu_g\\
=&\int_{B_{2\rho}}F_{\xi_i}(\nu) D_i(u_{\rho}\zeta w) dx
=-\int_{B_{2\rho}}\mathrm{div}(D_{\xi'}F(Du,-1))(u_{\rho}\zeta w) dx=0.\end{split}\]
Therefore,
\[\int_{C_{2\rho}\cap\{|u|<\rho\}}\left(F(\nu)-F_{\xi_{n+1}}(\nu)\nu_{n+1}\right)\zeta wd\mu_g+\int_{C_{2\rho}}u_{\rho}[w(\delta_{n+1}^F\zeta)+\zeta(\delta_{n+1}^Fw)]d\mu_g=0,\]
where $\{|u|<\rho\}:=\{x\in B_{3\rho}:|u(x)|<\rho\}.$ Taking into account that $\alpha<F(\nu)<\beta$, 
we obtain
\begin{equation}\label{e4.17}\begin{split}
\alpha \int_{C_{2\rho}\cap\{|u|<\rho\}}\zeta w d\mu_g\leq & \int_{C_{2\rho}\cap\{|u|<\rho\}}\zeta F_{\xi_{n+1}}(\nu)w\nu_{n+1}d\mu_g+2\rho\int_{C_{2\rho}\cap\{u>-\rho\}}w|\delta_{n+1}^F\zeta|d\mu_g\\
&+2\rho \int_{C_{2\rho}\cap\{u>-\rho\}} \zeta|\delta_{n+1}^Fw|d\mu_g,
\end{split}\end{equation}
where we used the fact that $|u_{\rho}|\leq 2\rho$.

For the first integral on the right hand side, since $|D_{\xi}F(\nu)|\leq c$, we obtain
\[
\int_{C_{2\rho}\cap\{|u|<\rho\}}\zeta F_{\xi_{n+1}}(\nu)w\nu_{n+1}d\mu_g=-\int_{B_{2\rho}\cap\{|u|<\rho\}}\zeta F_{\xi_{n+1}}(\nu)wdx\leq c\int_{B_{2\rho}\cap\{|u|<\rho\}}\zeta w dx.\]
Recalling $v \geq 1$ and $w = \ln v \geq 0$, we employ the algebraic inequality $c w \leq \frac{\alpha}{2} v w + c_{\alpha}$ to get
\[ c\int_{B_{2\rho}\cap\{|u|<\rho\}}\zeta w dx\leq  \frac{\alpha}{2}\int_{B_{2\rho}\cap\{|u|<\rho\}}\zeta w v dx+ c_{\alpha}\int_{B_{2\rho}\cap\{|u|<\rho\}}\zeta   dx \leq  \frac{\alpha}{2}\int_{C_{2\rho}\cap\{|u|<\rho\}}\zeta w d\mu_g+c_{\alpha}\rho^n,\]
which implies that
\begin{equation}\label{e4.18}\int_{C_{2\rho}\cap\{|u|<\rho\}}\zeta F_{\xi_{n+1}}(\nu)w\nu_{n+1}d\mu_g\leq \frac{\alpha}{2}\int_{C_{2\rho}\cap\{|u|<\rho\}}\zeta w d\mu_g+c_{\alpha}\rho^n
\end{equation}

Now, let $\tau(t)$ be a smooth function with support in the interval $[-2\rho,\sup_{B_{2\rho}}u+\rho]$ with $0\leq\tau\leq1$, $\tau\equiv 1$ in $[-\rho,\sup_{B_{2\rho}}u]$ and $|\tau'|\leq 2/\rho$.  Let $\psi(z)=\zeta(x)\tau(z_{n+1})$ and since $\zeta$ is independent of $z_{n+1}$, we have
\[|\delta_{n+1}^F\zeta|=|\nu_{n+1}F_{\xi_{i}}(\nu)\zeta_{i}|\leq c|\nu_{n+1}||D\zeta|\leq2c|\nu_{n+1}|\rho^{-1}.\]
Using the relation $-w\nu_{n+1}=w e^{-w}\leq e^{-1}$, we obtain
\begin{equation}\label{e4.19}
2\rho\int_{C_{2\rho}\cap\{u>-\rho\}}w|\delta_{n+1}^F\zeta|d\mu_g\leq 4c \int_{C_{2\rho}\cap\{u>-\rho\}}w|\nu_{n+1}|d\mu_g\leq4ce^{-1}\mu_g(C_{2\rho}\cap \mathrm{spt}\psi)
\end{equation}
and
\begin{equation}\label{e4.20}
\int_{C_{2\rho}\cap\{u>-\rho\}} \zeta|\delta_{n+1}^Fw|d\mu_g\leq\int_{C_{2\rho}}\psi|\delta_{n+1}^Fw|d\mu_g\leq\left[\mu_g\left(C_{2\rho}\cap\mathrm{spt}\psi\right)\right]^{1/2}\left(\int_{C_{2\rho}}\psi^2|\delta_{n+1}^F w|^2d\mu_g\right)^{1/2}.
\end{equation}

From the homogeneity of $F$, we find that
\[\left(F(\nu)e_{n+1}-\nu_{n+1}D_{\xi}F(\nu)\right)\cdot\nu=F(\nu)\nu_{n+1}-\nu_{n+1} F(\nu)=0.\]
Then  we have
\[\delta_{n+1}^F w=\nabla w\cdot\left(F(\nu)e_{n+1}-\nu_{n+1}D_{\xi}F(\nu)\right)=(\nabla w)^{T}\cdot\left(F(\nu)e_{n+1}-\nu_{n+1}D_{\xi}F(\nu)\right),\]
where $(\nabla w)^{T}=\nabla w-\langle\nabla w,\nu\rangle\nu$ and $\nu$ is the normal vector to the graph of $u$ .
Therefore,
\[|\delta_{n+1}^Fw|^2\leq\left|F(\nu)e_{n+1}-\nu_{n+1}D_{\xi}F(\nu)\right|^2 \left|(\nabla w)^{T}\right|^2\leq c_0 \left|(\nabla w)^{T}\right|^2,\]
where $c_0:=\sup_{\xi\in\mathbb{S}^n}|F(\xi)e_{n+1}-\xi_{n+1}D_{\xi}F(\xi)|^2<\infty$. Moreover, the ellipticity of $F$ gives that
\[\left|(\nabla w)^{T}\right|^2\leq \lambda^{-1}F_{\xi_i\xi_j}(\nu)  w_i w_j,\]
where we used  the Einstein summation convention for $i,j=1,\dots, n$.
Using the above inequalities, we obtain that
\begin{equation}\label{e4.21}
\int_{C_{2\rho}}\psi^2|\delta^F_{n+1}w|^2 d\mu_g\leq c_0\lambda^{-1}\int_{C_{2\rho}} \psi^2 F_{\xi_i\xi_j }(\nu) w_i w_j d\mu_g.
\end{equation}

We notice that the $(n+1)$-th component $\nu_{n+1}$ satisfies the following Jacobi equation
\[\operatorname{div}\left(F_{\xi_i\xi_j}(\nu)(\nabla  \nu_{n+1})_j\right)+|S_F|^2\nu_{n+1}=0,\]
where $S_F$ is the $F$-anisotropic shape operator on $S$. Since $\nabla  \nu_{n+1}=-\nu_{n+1}\nabla  w$, we have
\[\begin{split}
\operatorname{div}\left(F_{\xi_i\xi_j}(\nu)(\nabla  \nu_{n+1})_j\right)=&-\operatorname{div} \left(\nu_{n+1}F_{\xi_i\xi_j}(\nu)  w_j\right)\\
=&-\nu_{n+1}\operatorname{div} \left(F_{\xi_i\xi_j}(\nu)  w_j\right)-F_{\xi_i\xi_j}(\nu)(\nabla\nu_{n+1})_{i} w_j\\
=&-\nu_{n+1}\operatorname{div} \left(F_{\xi_i\xi_j}(\nu)  w_j\right)+\nu_{n+1}F_{\xi_i\xi_j}(\nu)  w_{i} w_j,
\end{split}\]
which implies that
\[0=-\nu_{n+1}\operatorname{div} \left(F_{\xi_i\xi_j}(\nu)  w_j\right)+\nu_{n+1}F_{\xi_i\xi_j}(\nu)  w_{i} w_j+|S_F|^2\nu_{n+1}.\]
Dividing by $\nu_{n+1}$ yields
\begin{equation}\label{e4.22}\begin{split}
\operatorname{div}\left(F_{\xi_i\xi_j}(\nu)  w_j\right)=&F_{\xi_i\xi_j}(\nu) w_{i}  w_j+|S_F|^2\\
\geq&F_{\xi_i\xi_j}(\nu) w_{i}  w_j.
\end{split}\end{equation}
Combining \eqref{e4.21} and \eqref{e4.22}, we conclude that
\[\begin{split}
\int_{C_{2\rho}} \psi^2 F_{\xi_i\xi_j }(\nu) w_{i}  w_j d\mu_g
\leq&\int_{C_{2\rho}}\psi^2\operatorname{div}_{S}\left(F_{\xi_i\xi_j}(\nu)  w_j\right)d\mu_g\\
=&-2\int_{C_{2\rho}}\psi F_{\xi_i\xi_j}(\nu) \psi_i  w_j d\mu_g\\
\leq&\frac{1}{2}\int_{C_{2\rho}} \psi^2 F_{\xi_i\xi_j }(\nu)w_i w_j d\mu_g+2\int_{C_{2\rho}} F_{\xi_i\xi_j }(\nu)\psi_i\psi_j d\mu_g,
\end{split}\]
which implies that
\[\int_{C_{2\rho}} \psi^2 F_{\xi_i\xi_j }(\nu) w_{i}  w_j d\mu_g\leq 4\int_{C_{2\rho}} F_{\xi_i\xi_j }(\nu)\psi_i\psi_j d\mu_g.\]
Using the property of  $F$ in \eqref{e4.1}, we get
\begin{equation}\label{e4.23}
\int_{C_{2\rho}} \psi^2 F_{\xi_i\xi_j }(\nu) w_{i}  w_j d\mu_g\leq 4\Lambda\int_{C_{2\rho}}|\nabla \psi|^2d\mu_g.
\end{equation}
Therefore, from \eqref{e4.21} and \eqref{e4.23}, we deduce that
\begin{equation}\label{e4.24}
\int_{C_{2\rho}}\psi^2|\delta^F_{n+1}w|^2 d\mu_g\leq c_0\lambda^{-1}\Lambda \int_{C_{2\rho}}|\nabla \psi|^2d\mu_g\leq c\rho^{-2}\mu_g\left(C_{2\rho}\cap\operatorname{spt}\psi\right).
\end{equation}
In conclusion, we have from \eqref{e4.17}-\eqref{e4.24} that
\begin{equation}\label{e4.25}
\int_{S_{\rho}}w d\mu_g\leq\int_{C_{2\rho}\cap\{|z_{n+1}|<\rho\}}\zeta wd\mu_g\leq c_3 \mu_g\left(C_{2\rho}\cap\operatorname{spt}\psi\right).
\end{equation}
It remains to estimate $\mu_g\left(C_{2\rho}\cap\operatorname{spt}\psi\right)$. Let $\gamma(x)$ be a function with support in $B_{3\rho}$ with $0\leq\gamma\leq1$, $\gamma\equiv1$ in $B_{2\rho}$ and $|D\gamma|\leq 2\rho^{-1}$. We have
\[\int_{C_{3\rho}}\delta_{n+1}^F[\gamma\max\{u+2\rho,0\}]d\mu_g=0,\]
and then
\[\int_{C_{3\rho}\cap\{z_{n+1}>-2\rho\}}\gamma\left(F(\nu)-\nu_{n+1}F_{\xi_{n+1}}(\nu)\right)d\mu_g+\int_{C_{3\rho}}\max\{u+2\rho,0\}\delta_{n+1}^F\gamma d\mu_g=0.\]
Arguing as above, we find that since $\operatorname{spt}\psi\subset B_{2\rho}\times(-2\rho,\infty)$,
\[\begin{split}
\mu_g\left(C_{2\rho}\cap\operatorname{spt}\psi\right)\leq& \alpha^{-1}\int_{C_{3\rho}}|\nu_{n+1}||F_{\xi_{n+1}}(\nu)|d\mu_g+(\sup_{B_{3\rho}}u+2\rho)\int_{C_{3\rho}}|\nu_{n+1}||F_{\xi_i}(\nu)D_i\gamma| d\mu_g\\
=&\alpha^{-1}\int_{B_{3\rho}}|F_{\xi_{n+1}}(\nu)|dx+(\sup_{B_{3\rho}}u+2\rho)\int_{B_{3\rho}}|F_{\xi_i}(\nu)D_i\gamma| dx\\
\leq& c_4\rho^{n}\left(1+\rho^{-1}\sup_{B_{3\rho}}u\right),
\end{split}\]
where $c_4$ is a constant independent of $\rho$ and the conclusion follows from \eqref{e4.25}.
\end{proof}

\medskip
Combining   lemma \ref{lemma4.1} and lemma \ref{lemma4.2}, and recalling that $w=\log\sqrt{1+|Du|^2}$, we obtain the following proposition.
\begin{proposition}\label{proposition4.1}(A priori estimate of the gradient)
Let $u$ be a solution of the anisotropic minimal surface equation in $B_{3\rho}$. Then
\begin{equation}\label{e4.26}
\sup_{S_{\rho/2}(x_0)}|Du|\leq \exp\left\{c_5\left(1+\frac{\sup_{B_{3\rho}(x_0)}u-u(x_0)}{\rho}\right)\right\}.
\end{equation}
In particular, $|Du(x_0)|$ is bounded by the right-hand side of \eqref{e4.26}.
\end{proposition}

\section{Proof of Theorem \ref{theorem1.2} and Theorem \ref{corollary4.1}}
Having established the crucial a priori gradient bounds, this section is devoted to the proofs of Theorem \ref{theorem1.2} and Corollary \ref{corollary4.1}. We first recall a classical regularity result for second-order elliptic equations (see, e.g.,\cite{GE}):

\begin{lemma}\label{lemma4.3}
Let $u\in W_{loc}^{1,2}(\Omega)$ be a solution of
\[\int_{\Omega}(a_{ij}(x)D_j u-f_i)D_i\varphi dx=0\quad\text{for any }\varphi\in C_0^{\infty}(\Omega)\]
with conditions
\[a_{ij}\xi_i\xi_j\geq\beta |\xi|^2,\quad\beta>0.\]
Suppose that the coefficients $a_{ij}$ and the functions $f_i$ are of class $C^{m,\alpha}(\Omega)$. Then $u\in C^{m+1,\alpha}(\Omega)$.
\end{lemma}

\begin{proof}[proof of Theorem \ref{theorem1.2}]
Let $\{\varphi_j\}$ be a sequence of functions in $ C^2(\mathbb{R}^n)$, converging uniformly to $\varphi$ on $\partial\Omega$. By Theorem \ref{theorem1.1} we know that there exist some $u_j\in C^{0,1}(\Omega)$ minimizes the anisotropic area $\mathcal{A}_F$ in $L(\Omega,\varphi_j)$.
From \eqref{e3.1}, we have
\[\sup_{\Omega}|u_j-u_k|\leq\sup_{\partial\Omega}|\varphi_j-\varphi_k|,\]
and therefore, $u_j$ converge uniformly to some $u$ in $\Omega$. Now, let $K\subset\Omega$ be a compact set. Since $u_j$ satisfy the anisotropic minimal surface equation \eqref{e4.15}, from proposition \ref{proposition4.1}, we deduce that
\[\sup_{K}|Du_j|\leq L,\]
where $L$ is a constant depends on $K$ but not on $j$.
Let $w=\partial u_j/\partial x_s$ ($s=1,\cdots,n$), then it satisfies the equation
\[\frac{\partial}{\partial x_i}\left(a_{ik}(x)\frac{\partial w}{\partial x_k}\right)=0,\]
where $a_{ik}=F_{\xi_i\xi_k}(Du_j,-1)$. For any $\zeta\in\mathbb{R}^n$ denote by $\zeta=(\zeta_1,\cdots,\zeta_n)\neq0$, we set $\widetilde{\zeta}=(\zeta_1,\cdots,\zeta_n,0)$, then by the ellipticity of $F$, we have
\[\begin{split}
\sum_{i,j=1}^{n}F_{\xi_i\xi_j}(Du_j,-1)\zeta_i\zeta_j=&\sum_{i,j=1}^{n+1}F_{\widetilde{\zeta}_i\widetilde{\zeta}_j}(Du_j,-1)\widetilde{\zeta}_i\widetilde{\zeta}_j\\
=&\frac{1}{\sqrt{1+|Du_j|^2}}\sum_{i,j=1}^{n+1}F_{\widetilde{\zeta}_i\widetilde{\zeta}_j}\left(\frac{(Du_j,-1)}{\sqrt{1+|Du_j|^2}}\right)\widetilde{\zeta}_i\widetilde{\zeta}_j\\
\geq&c_6|\widetilde{\zeta}|^2=c_6|\zeta|^2,
\end{split}\]
where $c_6$ is a positive constant depend on $L$ and $\lambda$. Therefore, the theory of uniformly elliptic equations (see lemma \ref{lemma4.3}) gives a bound for the derivatives of any order:
\[\sup_{K}|D^{m}u_j|\leq L(K,m).\]
In particular, we may conclude that $u\in C^2(\Omega)\cap C^0(\overline{\Omega})$ and that $u_j\to u$ in $C_{loc}^2(\Omega)$ so that $u$ is a solution of the anisotropic Dirichlet problem \eqref{e1.2}.

\end{proof}

To prove Theorem \ref{corollary4.1}, we need the following  Harnack's inequality.

\begin{lemma}\label{lemma4.4}(\cite{GE})
Let $u\in W_{loc}^{1,2}(\Omega)$ be a positive weak solution of
\[\partial_i\left(a_{ij}(x)D_j u\right)=0\quad\text{ in }B_R\]
with conditions
\[a_{ij}\xi_i\xi_j\geq\beta |\xi|^2,\quad\beta>0.\]
Then for any $\theta<1$, there exists a constant $c$ depending on $\theta$ but not on $u$, such that
\[\sup_{B_{\theta R}}u\leq c \inf_{B_{\theta R}}u.\]
\end{lemma}

\begin{proof}[proof of Theorem \ref{corollary4.1}]
Let $x_0\in\mathbb{R}^n$, we have
\[\sup_{B_\rho(x_0)}u\leq C(1+\rho+|x_0|).\]
Therefore, from \eqref{e4.26},
\[\sup_{S_{\rho/2}(x_0)}|Du|\leq\exp\left\{c_5(1+C)+c_5 \rho^{-1}(|u(x_0)|+C(1+|x_0|))\right\}.\]
Letting $\rho\to\infty$, we obtain
\[\sup_{\mathbb{R}^n}|Du|\leq \exp\{c_5(1+C)\}:=c_6.\]
Similarly, if we set $w=\partial u/\partial x_s$ ($s=1,\cdots,n$), then it satisfies the following uniformly elliptic equation
\begin{equation}\label{e4.28}\frac{\partial}{\partial x_i}\left(a_{ik}(x)\frac{\partial w}{\partial x_k}\right)=0,\end{equation}
where $a_{ik}=F_{\xi_i\xi_k}(Du,-1)$. Since $w$ is bounded, the non-negative function $z=w-\inf w$ satisfies the same equation \eqref{e4.28}. Therefore, the Harnack's inequality implies that
\[\sup_{B_\rho}z\leq c\inf_{B_\rho}z\]
for any $\rho>0$ with $c$ independent of $\rho$. Letting $\rho\to\infty$, we obtain $\sup_{\mathbb{R}^n}z=0$ and thus $w=const$, which indicates that $u$ is an affine function.
\end{proof}

\section{Direct Methods}
In this section, we use the direct method  in the calculus of variations to minimize the anisotropic area
\[\mathcal{A}_{F}(u,\Omega)=\int_{\Omega}F(Du,-1) \]
which is the anisotropic total variation of $u$ in $BV(\Omega)$ defined by \eqref{e2.2}, with prescribed values $\varphi(x)$ on  $\partial\Omega$.  For simplicity, throughout this section we shall omit the trace operator and simply write $u$ for $\operatorname{Tr}u$  on $\partial\Omega$.

\begin{lemma}\label{lemma5.1}
Let $\mathcal{C}_{R}^+=\mathcal{B}_{R}\times (0,R)$ where $\mathcal{B}_R\subset\mathbb{R}^{n-1}$ is a ball centered at the origin with radius $R$. Let $\varphi$ be a function in $L^1(\mathcal{B}_R)$ with compact support. Then, for any $\varepsilon>0$, there exists a function $v\in W^{1,1}(\mathcal{C}_R^+)$ with trace $\varphi$ on $\mathcal{B}_R$ such that
\begin{equation}\label{e5.1}
\int_{\mathcal{C}_R^+}|v|dx\leq\varepsilon\int_{\mathcal{B}_R}|\varphi|d\mathcal{H}^{n-1}
\end{equation}
and \begin{equation}\label{e5.2}
\int_{\mathcal{C}_R^+}F(Dv,0)dx\leq(1+\varepsilon)\int_{\mathcal{B}_R}F(e_n,0)|\varphi|d\mathcal{H}^{n-1},
\end{equation}
where $e_n=(0,\cdots,1)$.
\end{lemma}
\begin{proof}
Let $\{\varphi_k\}$ be a sequence of $C^{\infty}$-functions in $\mathcal{B}_R$ and $\varphi_k\to\varphi$ in $L^1(\mathcal{B}_R)$. Assume that $\varphi_0=0$,
\begin{equation}\label{e5.3}\int_{\mathcal{B}_R}|\varphi_k |d\mathcal{H}^{n-1}\leq 2\int_{\mathcal{B}_R}|\varphi |d\mathcal{H}^{n-1}\end{equation}
and
\begin{equation}\label{e5.4}\sum_{k=0}^{\infty}\int_{\mathcal{B}_R}|\varphi_k -\varphi_{k+1} |d\mathcal{H}^{n-1}\leq\left(1+\frac{\varepsilon}{2}\right)\int_{\mathcal{B}_R}|\varphi |d\mathcal{H}^{n-1}.\end{equation}
Let $\{t_k\}$ be a decreasing sequence, converging to $0$. We set
\[v(x)=v(z,t)=\begin{cases}
\frac{t-t_{k+1}}{t_k-t_{k+1}}\varphi_k(z)+\frac{t_k-t}{t_k-t_{k+1}}\varphi_{k+1}(z)&\text{if }t_{k+1}\leq t\leq t_k,\\
0&\text{if }t>t_0.
\end{cases}\]
Therefore, we have
\[\begin{split}
\int_{\mathcal{C}_R^+}F(Dv,0) =&\sum_{k=0}^{\infty}\int_{t_{k+1}}^{t_{k}}\int_{\mathcal{B}_R}F(Dv,0)d\mathcal{H}^{n-1}dt\\
\leq&\sum_{k=0}^{\infty}(t_k-t_{k+1})\int_{\mathcal{B}_R}\left[F(D\varphi_k,0,0)+F(D\varphi_{k+1},0,0)+\frac{|\varphi_k-\varphi_{k+1}|}{t_k-t_{k+1}}F(e_n,0)\right]d\mathcal{H}^{n-1}.
\end{split}\]
If we choose $\{t_k\}$ such that $4t_0<\varepsilon$ and
\[t_k-t_{k+1}\leq\frac{\varepsilon\int_{\mathcal{B}_R}F(e_n,0)|\varphi| d\mathcal{H}^{n-1}}{\int_{\mathcal{B}_R}\left[F(D\varphi_k,0,0)+F(D\varphi_{k+1},0,0)\right]d\mathcal{H}^{n-1}+1}2^{-k-2},\]
then from \eqref{e5.3}, \eqref{e5.4}, we obtain
\[\begin{split}
\int_{\mathcal{C}_R^+}F(Dv,0)&\leq \left(1+\frac{\varepsilon}{2}\right)\int_{\mathcal{B}_R}F(e_n,0)|\varphi|d\mathcal{H}^{n-1}+\varepsilon\sum_{k=0}^{\infty}2^{-k-2}\int_{\mathcal{B}_R}F(e_n,0)|\varphi| d\mathcal{H}^{n-1}\\
&=(1+\varepsilon)\int_{\mathcal{B}_R}F(e_n,0)|\varphi|d\mathcal{H}^{n-1}
\end{split}\]
and
\[\begin{split}
\int_{\mathcal{C}_R^+}|v| \leq&\sum_{k=0}^{\infty}\left(\int_{\mathcal{B}_R}|\varphi_k|d\mathcal{H}^{n-1}+\int_{\mathcal{B}_R}|\varphi_{k+1}|d\mathcal{H}^{n-1}\right)(t_k-t_{k+1})\\
\leq& 4\int_{\mathcal{B}_R}|\varphi|d\mathcal{H}^{n-1}\sum_{k=0}^{\infty}(t_k-t_{k+1})=4t_0\int_{\mathcal{B}_R}|\varphi|d\mathcal{H}^{n-1}<\varepsilon \int_{\mathcal{B}_R}|\varphi|d\mathcal{H}^{n-1}.
\end{split}\]
Then, we complete the proof.

\end{proof}

By a standard argument based on a partition of unity, we obtain the following proposition:
\begin{proposition}\label{proposition5.1}
Let $\Omega$ be a bounded open set with $C^1$ boundary $\partial\Omega$ and let $\varphi\in L^1(\partial\Omega)$. Then, for any $\varepsilon>0$, there exists a function $v\in W^{1,1}(\Omega)$ having trace $\varphi$ on $\partial\Omega$ such that
\begin{equation}\label{e5.5}
\int_{\Omega}|v| \leq\varepsilon \int_{\partial\Omega}|\varphi|d\mathcal{H}^{n-1},
\end{equation}
\begin{equation}\label{e5.6}
\int_{\Omega}F(Dv,0) \leq(1+\varepsilon) \int_{\partial\Omega}F(\nu,0)|\varphi|d\mathcal{H}^{n-1},
\end{equation}
where $\nu$ denotes the outward unit normal vector to $\partial\Omega$.
\end{proposition}

\begin{proposition}\label{proposition5.2}
Let $\Omega$ be a bounded open set with $C^1$ boundary $\partial\Omega$, and let $\varphi\in L^1(\partial\Omega)$. We have
\begin{equation}\label{e5.7}
\begin{split}&\inf\left\{\mathcal{A}_F(u,\Omega): u\in BV(\Omega), u=\varphi\text{ on }\partial\Omega\right\}\\
=&\inf\left\{\mathcal{A}_F(u,\Omega)+\int_{\partial\Omega}F(\nu,0)|u-\varphi|d\mathcal{H}^{n-1}: u\in BV(\Omega)\right\}.
\end{split}\end{equation}
\end{proposition}
\begin{proof}
Obviously, the left-hand side of \eqref{e5.7} is greater than or equal to its right-hand side. It therefore suffices to prove the reverse inequality.

Let $u\in BV(\Omega)$ and let $\varepsilon>0$. It follows from proposition \ref{proposition5.1} that there exists a function $w\in W^{1,1}(\Omega)$ with $w=\varphi-u$ on $\partial\Omega$ and
\[\int_{\Omega}F(Dw,0) \leq(1+\varepsilon)\int_{\partial\Omega}F(\nu,0)|u-\varphi|d\mathcal{H}^{n-1}.\]
The function $v=u+w$ belongs to $BV(\Omega)$ with $v=\varphi$ on $\partial\Omega$. Moreover,
\[\begin{split}
\int_{\Omega}F(Dv,-1) \leq&\int_{\Omega}[F(Du,-1)+F(Dw,0)] \\
\leq& \int_{\Omega}F(Du,-1) +(1+\varepsilon)\int_{\partial\Omega}F(\nu,0)|u-\varphi|d\mathcal{H}^{n-1}.
\end{split}\]
Letting $\varepsilon\to0$, \eqref{e5.7} follows at once.
\end{proof}




\begin{remark}\label{remark5.1}The above result suggests the following weaker form of the Dirichlet problem for anisotropic non-parametric minimal surfaces:
Given a function $\varphi$ in $L^1(\partial\Omega)$, find a function $u\in BV(\mathcal{B})$ minimizing the functional
\begin{equation}\label{e5.9}
\mathcal{E}_{F}(v,\Omega)=\int_{\Omega}F(Dv,-1) +\int_{\partial\Omega}F(\nu,0)|v-\varphi|d\mathcal{H}^{n-1}
\end{equation}
among all functions $v\in BV(\Omega)$.

 If $\mathcal{B}$ is a ball containing $\overline{\Omega}$, we can use proposition \ref{proposition5.1} to  extend $\varphi$ to a $W^{1,1}$ function in $\mathcal{B}\setminus\overline{\Omega}$ still denoted by $\varphi$.  If we set for $v\in BV( \Omega)$,
\[v_{\varphi}=\begin{cases}
v(x)&\text{for }x\in\Omega\\
\varphi(x)&\text{for }x\in\mathcal{B}\setminus\Omega.
\end{cases}\]
 Then, the function $v_{\varphi}\in BV(\mathcal{B})$ and
\[\begin{split}
\int_{\mathcal{B}}F(Dv_{\varphi},-1) =&\int_{\Omega}F(Dv,-1) +\int_{\mathcal{B}\setminus\overline{\Omega}}F(D\varphi,-1) +\int_{\partial\Omega}F(\nu,0)|v-\varphi|d\mathcal{H}^{n-1}\\
=&\mathcal{E}_{F}(v,\Omega)+\int_{\mathcal{B}\setminus\overline{\Omega}}F(D\varphi,-1) .
\end{split}\]
We have therefore an equivalent formulation of the anisotropic Dirichlet problem: Given a function $\varphi\in W^{1,1}(\mathcal{B}\setminus\overline{\Omega})$, find a function $u\in BV(\mathcal{B})$, coinciding with $\varphi$ in $\mathcal{B}\setminus\overline{\Omega}$ and minimizing the anisotropic area $\mathcal{A}_F(v,\mathcal{B})$ among all functions $v\in BV(\mathcal{B})$ with $v=\varphi$ in $\mathcal{B}\setminus\overline{\Omega}$.
\end{remark}

\medskip

Since the set of functions uniformly bounded in $BV$-norm is relatively compact in $L^1(\Omega)$ and the anisotropic area $\mathcal{A}_F$ is lower semicontinuous, we have:
\begin{proposition}\label{proposition5.3}
Let $\Omega$ be a bounded open set with Lipschitz  boundary $\partial\Omega$, and let $\varphi$ be a function in $L^1(\partial\Omega)$. Then the functional $\mathcal{E}_{F}(u,\Omega)$ attains its minimum in $BV(\Omega)$.
\end{proposition}

\medskip

We shall prove that any function $u$ minimizing $\mathcal{E}_F$ is regular in the following.

\begin{lemma}\label{lemma5.2}
Let $u\in BV(\Omega)$ and let
\[U=\{(x,t)\in\Omega\times\mathbb{R}: t<u(x)\}\]
be the subgraph of $u$. Then
\begin{equation}\label{e5.8}
\int_{\Omega}F(Du,-1) =\int_{\Omega\times\mathbb{R}}|D\chi_{U}|_F .
\end{equation}
\end{lemma}
\begin{proof}
Suppose that $u$ is bounded and $u\geq1$. Let $g_1(x),\dots,g_{n+1}(x)$ be functions with compact support in $\Omega$ with $F^o(g)\leq1$ in $\Omega$. Let $\sigma(t)$ support in $[0,\sup_{\Omega}u+1]$ such that $|\sigma|\leq1$  and $\sigma\equiv1$ in $[1,\sup_{\Omega}u]$. Set $\gamma(x,t)=g(x)\sigma(t)$ with $g=(g_1,\dots,g_{n+1})$, we have $F^o(\gamma(x,t))=|\sigma(t)|F^o(g(x))\leq 1$ in $\Omega\times\mathbb{R}$. Therefore,
\[\begin{split}
\int_{\Omega\times\mathbb{R}}|D\chi_{U}|_F  \geq&\int_{U}\operatorname{div}\gamma   =\int_{U}\sum_{i=1}^{n+1} D_i\gamma_i   \\
=&\int_{\Omega}dx\int_{0}^{u(x)}\left[g_{n+1}(x)\sigma'(t)+\sigma(t)\sum_{i=1}^{n}D_i g_i(x)\right]dt.
\end{split}\]
Since
\[\int_{0}^{u(x)}\sigma'(t)dt=\sigma(u(x))-\sigma(0)=1\]
and
\[\int_0^{u(x)}\sigma(t)dt=\int_1^{u(x)}\sigma(t)dt+\int_0^1\sigma(t)dt=u(x)-c\]
with $c=1-\int_0^1\sigma(t)dt\geq0$, we have
\[\int_{\Omega\times\mathbb{R}}|D\chi_{U}|_Fdxdt\geq\int_{\Omega}g_{n+1}dx+\int_{\Omega}\operatorname{div}g (u(x)-c)dx=\int_{\Omega}[g_{n+1}+u\operatorname{div}g]dx,\]
which implies that
\[\int_{\Omega\times\mathbb{R}}|D\chi_{U}|_F\geq\int_{\Omega}F(Du,-1) .\]

On the other hand, let $u_j\in C^{\infty}(\Omega)$, $u_j\to u$ in $L^1(\Omega)$ and $\int_{\Omega}F(Du_j,-1) \to\int_{\Omega}F(Du,-1) $. We have $U_j\to U$ in $L_{loc}^1(\Omega\times\mathbb{R})$ and thus
\[\int_{\Omega\times\mathbb{R}}|D\chi_{U}|_F \leq\liminf_{j\to\infty}\int_{\Omega\times\mathbb{R}}|D\chi_{U_j}|_F =\lim_{j\to\infty}\int_{\Omega}F(Du_j,-1) =\int_{\Omega}F(Du,-1) .\]
Therefore, \eqref{e5.8} holds for any bounded $u\in BV(\Omega)$.

Finally, if $u$ is unbounded, writing
\[u_T(x)=\begin{cases}
T&\text{if } u\geq T\\
u(x)&\text{if }|u|<T\\
-T&\text{if }u\leq -T
\end{cases}\]
and letting $T\to\infty$, we get the desired result.

\end{proof}

\medskip

Now, we show that given a set $V$, we can decrease its anisotropic perimeter by replacing it with a suitable subgraph.
\begin{lemma}\label{lemma5.3}
Let $V\subset\Omega\times\mathbb{R}:=Q$ be a measurable set, and let
\[\Omega\times(-\infty,-T)\subset V\subset\Omega\times(-\infty,T)\]
for some $T>0$. For $x\in\Omega$, let
\[w(x)=\lim_{k\to\infty}\left(\int_{-k}^k \chi_{V}(x,t)dt-k\right).\]
Then,
\begin{equation}\label{e5.10}
\int_{\Omega}F(Dw,-1) \leq\int_{\Omega\times\mathbb{R}}|D\chi_{V}|_F  .
\end{equation}
\end{lemma}
\begin{proof}
Clearly,
\[\partial V\cap Q\subset\Omega\times(-T,T).\]
Setting
\[w_k=\int_{-k}^k \chi_{V}(x,t)dt-k,\]
and notice that $w_k=w_l$ for $k,l\geq T$, we have that $w(x)$ is a bounded measurable function and $-T\leq w(x)\leq T$.

Let $g(x)=(g_1(x),\dots,g_{n+1}(x))\in C_c^1(\Omega;\mathbb{R}^{n+1})$ with $F^o(g)\leq 1$, and let $\eta(t)$ be a smooth function such that $0\leq\eta\leq 1$ and
\[\eta(t)=\begin{cases}
0&\text{if }|t|\geq T+1\\
1&\text{if }|t|\leq T.
\end{cases}\]
We have
\[\int_{-\infty}^{\infty}\eta'(t)\chi_{V}(x,t)dt=\int_{-T-1}^{-T}\eta'(t)dt=\eta(-T)-\eta(-T-1)=1\]
and
\[\int_{-\infty}^{\infty}\eta(t)\chi_{V}(x,t)dt=\int_{-T}^{T}\chi_{V}(x,t)dt+\int_{-T-1}^{-T}\eta(t)dt=w+\alpha,\]
where $\alpha=T+\int_{-T-1}^{-T}\eta(t) dt\geq0$.

Therefore,
\[\begin{split}\int_{Q}|D\chi_{V}|_F \geq&\int_{Q}\chi_{V}(x,t)\sum_{i=1}^{n+1}\frac{\partial}{\partial x_i}\left(\eta(t)g_i(x)\right)dxdt\\
=&\int_{\Omega}g_{n+1}+(w +\alpha)\sum_{i=1}^{n}\frac{\partial g_i}{\partial x_i} dx=\int_{\Omega}(g_{n+1}+w\operatorname{div}g)dx.
\end{split}\]
Taking the supremum over $g$, we then get the conclusion.

\end{proof}

\medskip
In fact, the restriction that $\partial V\cap Q$ is bounded can be removed.
\begin{lemma}\label{lemma5.4}
Let $V\subset Q$ be a measurable set  in $Q$, and assume that
\begin{itemize}\item[(i)]for almost every $x\in\Omega$,
\[\lim_{t\to+\infty}\chi_{V}(x,t)=0\quad\text{and}\quad \lim_{t\to-\infty}\chi_{V}(x,t)=1;\]
\item[(ii)]  The symmetric difference $(V\setminus Q^-)\cup(Q^-\setminus V)$ has finite measure, where $Q^- = \Omega \times (-\infty, 0)$.
\end{itemize}
Then,  the function
\[w(x)=\lim_{k\to\infty}\left(\int_{-k}^k \chi_{V}(x,t)dt-k\right)\]
belongs to $L^1(\Omega)$ and
\[\int_{\Omega}F(Dw,-1) \leq\int_{Q}|D\chi_{V}|_F .\]
\end{lemma}
\begin{proof}
It follows from (i) that the sequence $w_k$ converges to $w$ almost everywhere in $\Omega$.
To establish the $L^1(\Omega)$ convergence, we define the dominating function
\[H(x) = \int_{-\infty}^\infty |\chi_V(x,t) - \chi_{Q^-}(x,t)| dt.\]
 Since the set $(V\setminus Q^-)\cup(Q^-\setminus V)$ has finite measure, we have $H \in L^1(\Omega)$. For any $k > 0$, we observe that
\[|w_k(x)| = \left| \int_0^k \chi_V(x,t) dt - \int_{-k}^0 (1-\chi_V(x,t)) dt \right| \leq H(x) \quad \text{for a.e. } x \in \Omega.\]
By Lebesgue's Dominated Convergence Theorem, it follows that $w \in L^1(\Omega)$ and $w_k \to w$ in $L^1(\Omega)$ as $k \to \infty$.

Now, let \[V_k=V\cup\left(\Omega\times(-\infty,-k)\right)\setminus\left(\Omega\times(k,+\infty)\right),\] then $\Omega\times(-\infty,-k)\subset V_k\subset \Omega\times(-\infty,k)$ and from lemma \ref{lemma5.3}, we get
\[\begin{split}\int_{\Omega}F(Dw_k,-1) \leq\int_{\Omega\times\mathbb{R}}|D\chi_{V_k}|_F \leq&\int_{\Omega\times\mathbb{R}}|D\chi_{V}|_F \\
&+F(0',1)\left[\int_{\Omega }\chi_{V}(x,k)dx+\int_{\Omega }(1-\chi_{V}(x,-k))dx \right].\end{split}\]
Then, the conclusion follows from (i) and the lower semicontinuity of the anisotropic area.

\end{proof}

\begin{lemma}\label{lemma5.5}
Let $u\in BV_{loc}(\Omega)$ be a local minimum of the anisotropic area functional $\mathcal{A}_{F}$. Then, the set
\[U=\{(x,t)\in\Omega\times\mathbb{R}:t<u(x)\}\]
minimizes locally the perimeter in $Q=\Omega\times\mathbb{R}$.
\end{lemma}
\begin{proof}
Let $A\subset\subset \Omega$ and let $V$ be a measurable set that has locally finite perimeter in $Q$, coinciding with $U$ outside some compact set $K\subset A\times\mathbb{R}$. It is obvious that $U$ and therefore $V$ satisfies the assumption in lemma \ref{lemma5.4}. The function $w$ coincides with $u$ outside $A$ and it follows from lemma \ref{lemma5.2} and lemma \ref{lemma5.4} that
\[\int_{A\times\mathbb{R}}|D\chi_{U}|_F  =\int_{A}F(Du,-1) \leq \int_{A}F(Dw,-1) \leq\int_{A\times\mathbb{R}}|D\chi_{V}|_F .\]
\end{proof}

\begin{lemma}\label{lemma5.7}
Let $u\in BV_{loc}(\Omega)$ minimize the anisotropic area functional $\mathcal{A}_F$, then $u$ is locally bounded in  $\Omega$.
\end{lemma}
\begin{proof}

Assume by contradiction that there exists a compact set $K\subset\Omega$ such that $\operatorname{ess}\sup_{K}u=\infty$.
Let
$R_0 = \frac{1}{2}\operatorname{dist}(K,\partial\Omega)$. For any $M\in\mathbb{N}$, there exists a point $x\in K$  be a Lebesgue point of $u$ such that $u(x)>2MR_0$. It follows that the points $z_i=(x,2iR_0)$ for $i=0,1,\dots,M$ belong to $U$. For any $R \in (0, R_0)$, we denote by $\mathcal{B}_R(z_i)$ the ball in $\mathbb{R}^{n+1}$ centered at $z_i$ with radius $R$.

Since $U$ locally minimizes the anisotropic perimeter in $Q$,
 we obtain
\[\int_{\mathcal{B}_R(z_i)}|D\chi_{U}|_F \leq \int_{\partial \mathcal{B}_R(z_i)}F(\nu)\chi_{U} d\mathcal{H}^{n}\]
for each $i = 0, 1, \dots, M$, and for almost all $R \in (0, R_0)$. Moreover, for almost all $R \in (0, R_0)$, we have
\[\int_{Q}|D\chi_{U\cap \mathcal{B}_R(z_i)}|_F = \int_{\mathcal{B}_R(z_i)}|D\chi_{U}|_F + \int_{\partial \mathcal{B}_R(z_i)} F(\nu)\chi_{U} d\mathcal{H}^{n}.\]
Using the fact that$F(-\nu) \leq \beta$ and $F(\nu) \leq \beta$, and applying the coarea formula, we get
\begin{equation}\label{e5.14}
\int_{Q}|D\chi_{U\cap \mathcal{B}_R(z_i)}|_F \leq \int_{\partial \mathcal{B}_R(z_i)} (F(-\nu) + F(\nu))\chi_{U} d\mathcal{H}^{n} \leq 2\beta \frac{d}{dR}|U\cap \mathcal{B}_R(z_i)|.\end{equation}
According to the anisotropic isoperimetric inequality established in \cite{AFT}, we have
\[\int_{Q}|D\chi_{U\cap \mathcal{B}_R(z_i)}|_F \geq (n+1)\kappa_{n+1}^{1/(n+1)}|U\cap \mathcal{B}_R(z_i)|^{1-1/(n+1)},\]
where $\kappa_{n+1}$ is the volume of the unit Wulff ball centered at the origin.
Combining this with
\eqref{e5.14} yields
\[\frac{d}{dR}|U\cap \mathcal{B}_R(z_i)| \geq \frac{(n+1)\kappa_{n+1}^{1/(n+1)}}{2\beta}|U\cap \mathcal{B}_R(z_i)|^{1-1/(n+1)}.\]
Integrating this differential inequality with respect to $R$, we obtain
\[|U\cap \mathcal{B}_R(z_i)| \geq \left( \frac{\kappa_{n+1}^{1/(n+1)} R}{2\beta} \right)^{n+1} = \frac{\kappa_{n+1}}{(2\beta)^{n+1}} R^{n+1} := c(n) R^{n+1}.\]

Now, let $K_R=\{y\in \Omega : \operatorname{dist}(y,K)<R\}$. Since $R < R_0$, we have $K_R \subset\subset \Omega$ and the projection of $\mathcal{B}_R(z_i)$ onto $\Omega$ is contained in $K_R$. Since $z_i = (x, 2iR_0)$ and $R < R_0$, for any $i \geq 1$, the balls $\mathcal{B}_R(z_i)$ lie entirely in the upper half-space $\{t > 0\}$. Consequently, if $(y,t) \in U \cap \mathcal{B}_R(z_i)$, then $0 < t < u(y)$, which implies $u(y) > 0$. Using the disjointness of these balls, we obtain
\[\int_{K_R}|u|dy \geq \int_{K_R} u^+ dy \geq \sum_{i=1}^{M}|U\cap \mathcal{B}_R(z_i)| \geq c(n) M R^{n+1}.\]
Since $M$ is arbitrary, this implies $\int_{K_R}|u|dy = \infty$, which contradicts the fact that $u \in BV(K_R)$.

To exclude $\inf_K u = -\infty$, we apply the same density argument to the complement $Q \setminus U$. Since $F$ is even, $Q \setminus U$ also locally minimizes the anisotropic perimeter. This gives a local lower bound for $u$. Hence, $u$ is locally bounded in $\Omega$.

\end{proof}

\medskip

A second consequence of lemma \ref{lemma5.5} is that the boundary of $U$, i.e. $\partial U$ is a regular hypersurface outside a closed set $\Sigma$ with $\mathcal{H}^{n-2}(\Sigma)=0$ (see \cite{ASS}). Now, we show that the function $u$ is regular in the set
\[Reg:=\Omega\setminus\operatorname{proj}\Sigma.\]
To prove this, it is sufficient to show that $\nu_{n+1}<0$ on $\partial U\setminus\Sigma$.
Suppose on the contrary that $\nu_{n+1}(x_0)=0$ for some $x_0\in\partial U\setminus\Sigma$. By the implicit function theorem, we can write $x_0 = (x_{0}^1, y_0)$ and represent $\partial U$ as the graph of a smooth function in a neighborhood of $x_0$:
\[x_1=v(x_2,\dots, x_{n+1})\quad\text{with }\frac{\partial v}{\partial x_{n+1}}(y_0)=0\text{ and }\frac{\partial v}{\partial x_{n+1}}\leq0.\]
Then, the function $v$ is of course a solution of the anisotropic minimal surface equation
\[\sum_{i=2}^{n+1}\frac{\partial}{\partial x_i}\left(F_{\xi_i}(1,-Dv)\right)=0.\]
As in lemma \ref{lemma4.3},  $w=\frac{\partial v}{\partial x_{n+1}}\leq0$ is a solution of the uniformly elliptic equation
\[\sum_{i,j=2}^{n+1}\frac{\partial}{\partial x_i}\left(F_{\xi_i\xi_{j}}(1,-Dv)\frac{\partial w}{\partial x_{j}}\right)=0.\]
Therefore, we can apply the following strong maximum principle:
\medskip
\begin{lemma}\label{lemma5.8}(Strong maximum principle): Let $\overline{v}$ and $\underline{v}$ be a supersolution and a subsolution of
\[\operatorname{div}A(x,Du)=0\quad\text{in }\Omega\]
with \[\sigma |\zeta|^2\leq\sum_{i,j=1}^{n}\frac{\partial A_i}{\partial p_j}\zeta_i\zeta_j\leq \gamma |\zeta|^2\quad\text{for any }\zeta\in\mathbb{R}^n.\]
If $\overline{v}\geq \underline{v}$ on $\partial\Omega$ and $\overline{v}(x_0)=\underline{v}(x_0)$ for some $x_0\in\Omega$, then $\overline{v}=\underline{v}$.
\end{lemma}
\medskip

Then, we conclude that $\frac{\partial v}{\partial x_{n+1}}=0$ and thus $\nu_{n+1}$ vanishes identically in a neighborhood $N$ of $x_0$. Setting $\Gamma=\operatorname{proj}N$, we have $\mathcal{H}^{n-1}(\Gamma)>0$.
For any $y \in \Gamma \setminus \operatorname{proj}\Sigma$, the vertical line $\{y\} \times \mathbb{R}$ does not intersect $\Sigma$, which forces $\{y\} \times \mathbb{R} \subset \partial U$. This contradicts the local boundedness of $u$. Thus, $\Gamma \subset \operatorname{proj}\Sigma$. Since the projection is Lipschitz, we obtain
\[
\mathcal{H}^{n-1}(\Sigma) \geq \mathcal{H}^{n-1}(\operatorname{proj}\Sigma) \geq \mathcal{H}^{n-1}(\Gamma) > 0,
\]
which contradicts $\mathcal{H}^{n-2}(\Sigma) = 0$. Hence, $\nu_{n+1} < 0$ on $\partial U \setminus \Sigma$.



\medskip

\begin{lemma}\label{lemma5.9}
Let $u\in BV_{loc}(\Omega)$ minimize the anisotropic area functional $\mathcal{A}_F$ in $\Omega$. Then, $u\in W_{loc}^{1,1}(\Omega)$.
\end{lemma}
\begin{proof}
Let $S=\operatorname{proj}\Sigma$. Since $\mathcal{H}^{n-2}(\Sigma)=0$, we have $\mathcal{H}^{n-2}(S)=0$, which implies the Lebesgue measure $|S|=0$.
For any open set $A\subset\subset\Omega$, we have
\[
\int_{A}F(Du,-1) = \int_{A\setminus S}F(Du,-1) + \int_{S\cap A}F(D^s u, 0),
\]
where $D^s u$ is the singular part of the Radon measure $Du$.

On the other hand, the anisotropic perimeter of the subgraph $U$ satisfies
\[
\int_{A}F(Du,-1) = P_F(U, A\times\mathbb{R}) = \int_{\partial^* U \cap (A\times\mathbb{R})} F(\nu^{U}) d\mathcal{H}^n.
\]
Since $\nu_{n+1} < 0$ on $\partial U \setminus \Sigma$, the projection map $\operatorname{proj}$ is a local diffeomorphism on $\partial U \setminus \Sigma$. Since $|S| = 0$ and $\mathcal{H}^n(\Sigma) = 0$, we have $\mathcal{H}^n(\partial^* U \cap (S \times \mathbb{R})) = 0$. Consequently,
\[
\int_{A}F(Du,-1) = P_F(U, (A \setminus S)\times\mathbb{R}) = \int_{A\setminus S}F(Du,-1).
\]
Comparing the two decompositions yields
\[
\alpha \int_{S\cap A} |D^s u| \leq \int_{S\cap A}F(D^s u, 0) = 0.
\]
Thus, the singular part $D^s u$ vanishes in $A$. Since $A\subset\subset\Omega$ is arbitrary, we conclude that $Du \in L^1_{loc}(\Omega)$, and hence $u\in W_{loc}^{1,1}(\Omega)$.

\end{proof}

\medskip
Since the anisotropic area functional $\mathcal{A}_F$ is strictly convex in $W^{1,1}(\Omega)$, we have:
\begin{proposition}\label{proposition5.4}
Let $\Omega$ be connected and let $\varphi\in L^1(\partial\Omega)$. If $u$ and $v$ are two minimizers of the functional
\[\int_{\Omega}F(Du,-1)dx+\int_{\partial\Omega}F(\nu,0)|u-\varphi|d\mathcal{H}^{n-1},\]
then $v=u+c$ for some constant $c \in \mathbb{R}$.
\end{proposition}
\begin
{proof}
By the convexity of the boundary term and the strict convexity of $F(p,-1)$ with respect to $p$, we deduce  that $Du = Dv$ a.e. in $\Omega$. Since $\Omega$ is connected, we conclude that $u - v$ is constant.
\end{proof}
\medskip

Finally, we can prove the regularity theorem.
\begin{proof}[proof of Theorem \ref{theorem1.3}] The existence of a minimizer follows directly from Proposition~\ref{proposition5.3}. Thus, it remains to establish the Lipschitz regularity of the minimizer.
Let $B_R$ be a ball in $\Omega$, then for any $w\in BV(B_R)$, it follows from proposition \ref{proposition5.2} that
\[\int_{B_R}F(Du,-1) \leq\int_{B_R}F(Dw,-1) +\int_{\partial B_R}F(\nu,0)|w-u|d\mathcal{H}^{n-1}.\]
Since the singular set $S$ satisfies $\mathcal{H}^{n-2}(S)=0$, there exists a sequence of open sets $\{S_j\}$ such that
\[S_{j+1}\subset\subset S_j,\quad \cap_{j\in\mathbb{N}}S_j=S\]
and\[\mathcal{H}^{n-1}(S_j\cap\partial B_R)\to0.\]
Let $\varphi_j$ be a smooth function on $\partial B_R$ such that
\[\varphi_j=u\text{ in }\partial B_R\setminus S_j\quad\text{and}\quad \sup_{\partial B_R}|\varphi_j|\leq 2\sup_{\partial B_R}|u|.\]
Observe that the ball $B_R$ is strictly convex, and hence its boundary $\partial B_R$ has   non-negative anisotropic mean curvature. Therefore, by  theorem \ref{theorem1.1}, we can assume that $u_j$ be the unique solution of the anisotropic Dirichlet problem with boundary datum $\varphi_j$ on $\partial B_R$. Moreover, the function $u_j$ are smooth in $ {B}_R$ and
\[\sup_{B_R}|u_j|\leq 2\sup_{\partial B_R}|u|.\]
Therefore, for any $w\in BV(B_R)$, we have
\begin{equation}\label{e5.12}
\int_{B_R}F(Du_j,-1) \leq\int_{B_R}F(Dw,-1) +\int_{\partial B_R}F(\nu,0)|w-\varphi_j|d\mathcal{H}^{n-1}.
\end{equation}

From the priori estimate of the gradient (see proposition \ref{proposition4.1}), we deduce that $|Du_j|$ are equibounded in every compact set $K\subset\ B_R$. By the Arzela-Ascoli theorem, there exists a subsequence (still denote by $u_j$) that will converge uniformly on compact subsets of $B_R$ to a Lipschitz-continuous function $v$. Taking $w=0$ in \eqref{e5.12}, we have
\[\int_{B_R}F(Du_j,-1) \leq F(0',1)|B_R|+\Lambda\int_{\partial B_R}|\varphi_j|d\mathcal{H}^{n-1}\leq c,\]
and thus $v\in W^{1,1}(B_R)$.

Now, we show that $v$ has trace $u$ on $\partial B_R$. Let $y\in\partial B_R$ be a regular point for $u$. Then, for $j$ large enough, $y\in\partial B_R\setminus S_j$ and for all $k> j$, $\varphi_k=u$ in a neighborhood of $y$ in $\partial B_R$. Therefore, we can construct two functions $\varphi_+$ and $\varphi_-$ of class $C^2$ on $\partial B_R$, such that
\begin{itemize}
\item[($1$)] $\varphi_{\pm}=u$ in a neighborhood of $y$ in $\partial B_R$,
\item[($2$)] $\varphi_-\leq \varphi_k\leq\varphi_+$ in $\partial B_R$ for any $k>j$.
\end{itemize}
Let $u^{\pm}$ be the solutions of the anisotropic Dirichlet problem with boundary datum $\varphi^{\pm}$ respectively. From the anisotropic weak maximum principle ( see lemma \ref{lemma3.1}), we get
\[u^-\leq u_k\leq u^+\quad\text{for any }k>j,\]
and then we obtain
\[u^-\leq v\leq u^+.\]
Therefore, we conclude that $v=u$ at every regular point $y\in\partial B_R$ in the trace sense and due to $\mathcal{H}^{n-1}(S)=0$, we get $v=u$ on $\partial B_R$.

Passing to the limit in \eqref{e5.12}, we have
\begin{equation}\label{e5.13}
\int_{B_R}F(Dv,-1) \leq\int_{B_R}F(Dw,-1) +\int_{\partial B_R}F(\nu,0)|w-u|d\mathcal{H}^{n-1}.
\end{equation}
Since $v=u$ on $\partial B_R$, the function $v$ also minimizes the functional on the right-hand side of \eqref{e5.13}. From proposition \ref{proposition5.4}, $v=u+c$ for some constant $c \in \mathbb{R}$ and since $v=u$ on $\partial B_R$, we finally get $v=u$. Therefore, we conclude that $u$ is Lipschitz continuous in $\Omega$.

\end{proof}

\subsection*{Acknowledgements}
The authors would like to thank Professor Antonio De Rosa for  his valuable suggestions regarding the sharpness of the one-sided linear growth condition in the Bernstein-type result. We also thank Dr. Ling Wang for raising insightful questions that led to a correction in the proof of the interior gradient estimates.

\end{document}